\documentclass[twoside]{amsart}
\usepackage{amsmath}
\usepackage{amssymb}
\usepackage[all]{xy}
\usepackage{graphicx}
\usepackage{mathrsfs}
\usepackage{rank-2-roots}
\usepackage{stmaryrd}
\usepackage{mathabx}
\usepackage{tikz-cd}
\usepackage[table]{colortbl}
\usetikzlibrary{decorations.text,calc,arrows.meta}
\usetikzlibrary{positioning}

\numberwithin{equation}{section}

\usepackage{url}

\usepackage[latin1]{inputenc}
\usepackage{xspace,amssymb,amsfonts,euscript}
\usepackage{amsthm,amsmath}
\usepackage{palatino}
\usepackage{euscript}
\input xy \xyoption {all}

\usepackage{tikz}
\usetikzlibrary{patterns,snakes}

\RequirePackage{color}
\definecolor{myred}{rgb}{0.75,0,0}
\definecolor{mygreen}{rgb}{0,0.5,0}
\definecolor{myblue}{rgb}{0,0,0.65}

\makeatletter

\def\chaptermark#1{}

\def\chapter{%
  \if@openright\cleardoublepage\else\clearpage\fi
  \thispagestyle{plain}\global\@topnum\z@
  \@afterindenttrue \secdef\@chapter\@schapter}

\def\@chapter[#1]#2{\refstepcounter{chapter}%
  \ifnum\c@secnumdepth<\z@ \let\@secnumber\@empty
  \else \let\@secnumber\thechapter \fi
  \typeout{\chaptername\space\@secnumber}%
  \def\@toclevel{0}%
  \ifx\chaptername\appendixname \@tocwriteb\tocappendix{chapter}{#2}%
  \else \@tocwriteb\tocchapter{chapter}{#2}\fi
  \chaptermark{#1}%
  \addtocontents{lof}{\protect\addvspace{10\p@}}%
  \addtocontents{lot}{\protect\addvspace{10\p@}}%
  \@makechapterhead{#2}\@afterheading}
\def\@schapter#1{\typeout{#1}%
  \let\@secnumber\@empty
  \def\@toclevel{0}%
  \ifx\chaptername\appendixname \@tocwriteb\tocappendix{chapter}{#1}%
  \else \@tocwriteb\tocchapter{chapter}{#1}\fi
  \chaptermark{#1}%
  \addtocontents{lof}{\protect\addvspace{10\p@}}%
  \addtocontents{lot}{\protect\addvspace{10\p@}}%
  \@makeschapterhead{#1}\@afterheading}
\newcommand\chaptername{Chapter}

\def\@makechapterhead#1{\global\topskip 7.5pc\relax
  \begingroup
  \fontsize{\@xivpt}{18}\bfseries\centering
    \ifnum\c@secnumdepth>\m@ne
      \leavevmode \hskip-\leftskip
      \rlap{\vbox to\z@{\vss
          \centerline{\normalsize\mdseries \uppercase\@xp{\chaptername}\enspace\thechapter}
          \vskip 3pc}}\hskip\leftskip\fi
     #1\par \endgroup
  \skip@34\p@ \advance\skip@-\normalbaselineskip
  \vskip\skip@ }
\def\@makeschapterhead#1{\global\topskip 7.5pc\relax
  \begingroup
  \fontsize{\@xivpt}{18}\bfseries\centering
  #1\par \endgroup
  \skip@34\p@ \advance\skip@-\normalbaselineskip
  \vskip\skip@ }
\def\appendix{\par
  \c@chapter\z@ \c@section\z@
  \let\chaptername\appendixname
  \def\thechapter{\@Alph\c@chapter}}

\newcounter{chapter}

\newif\if@openright

\makeatother

\renewcommand*\thechapter{\Roman{chapter}}

\RequirePackage{ifpdf}
\ifpdf
  \IfFileExists{pdfsync.sty}{\RequirePackage{pdfsync}}{}
  \RequirePackage[pdftex,
   colorlinks = true,
   urlcolor = myblue, 
   citecolor = mygreen, 
   linkcolor = myred, 
   pagebackref,
   bookmarksopen=true]{hyperref}
\else
  \RequirePackage[hypertex]{hyperref}
\fi

\RequirePackage{ae, aecompl, aeguill} 

    \def\CM{{\mathbb{C}}}

    \def\PM{{\mathbb{P}}}
    
    \def\RM{{\mathbb{R}}}

    \def\ZM{{\mathbb{Z}}}

\def\z{\zeta}

\newcommand{\nc}{\newcommand} \newcommand{\renc}{\renewcommand}

\newcommand{\rdots}{\mathinner{ \mkern1mu\raise1pt\hbox{.}
    \mkern2mu\raise4pt\hbox{.}
    \mkern2mu\raise7pt\vbox{\kern7pt\hbox{.}}\mkern1mu}}

\def\to{\rightarrow}

\def\longto{\longrightarrow}

\nc{\triright}{\stackrel{[1]}{\to}}
\nc{\longtriright}{\stackrel{[1]}{\longto}}

\nc{\Hb}{H^\bullet}

\nc{\Br}{\mathcal{B}}
\nc{\HotRR}{{}_R\mathcal{K}_R}
\nc{\HotR}{\mathcal{K}_R}
\nc{\excise}[1]{}
\nc{\defect}{\text{df}}
\nc{\h}[1]{\underline{H}_{#1}}

\nc{\Ga}{\mathbb{G}_a} 
\nc{\Gm}{\mathbb{G}_m} 

\nc{\Perv}{{\mathbf{P}}}

\nc{\IH}{{\mathrm{IH}}}

\nc{\ic}{\mathbf{IC}}

\nc{\gl}{{\mathfrak{gl}}}
\renc{\sl}{{\mathfrak{sl}}}
\renc{\sp}{{\mathfrak{sp}}}

\renc{\Im}{\textrm{Im}}

\nc{\HBM}{H^{BM}}

 \DeclareMathOperator{\Hom}{Hom}

\renewenvironment{proof}{{\it Proof.}}

\usepackage{etoolbox}
\makeatletter
\patchcmd{\@thm}{\thm@headfont{\scshape}}{\thm@headfont{\scshape\bfseries}}{}{}
\patchcmd{\@thm}{\thm@notefont{\fontseries\mddefault\upshape}}{}{}{}
\makeatother

\theoremstyle{definition}
\newtheorem{thm}{Theorem}[section]

\newtheorem{prop}[thm]{Proposition}
\newtheorem{cor}[thm]{Corollary}
  
\newtheorem{conj}[thm]{Conjecture}
\newtheorem{defi}[thm]{Definition}
\newtheorem{Eg}[thm]{Example}
\newtheorem{ex}[thm]{Examples}
\newtheorem{exercise}[thm]{Exercise}
\newtheorem{assumption}[thm]{Assumption}
\newtheorem{application}[thm]{Application}

\theoremstyle{remark}
\newtheorem{notation}[thm]{Notation}
\newtheorem{remark}[thm]{Remark}
\newtheorem{remarks}[thm]{Remarks}

\DeclareMathOperator{\Ext}{Ext}

\nc{\simto}{\stackrel{\sim}{\to}}
\nc{\simfrom}{\stackrel{\sim}{\leftarrow}}

\newcommand{\figref}[1]{\hyperref[#1]{Figure \ref{#1}}}
\newcommand{\lemref}[1]{\hyperref[#1]{Lemma \ref{#1}}}
\newcommand{\thmref}[1]{\hyperref[#1]{Theorem \ref{#1}}}
\newcommand{\conjref}[1]{\hyperref[#1]{Conjecture \ref{#1}}}
\newcommand{\propref}[1]{\hyperref[#1]{Proposition \ref{#1}}}
\newcommand{\corref}[1]{\hyperref[#1]{Corollary \ref{#1}}}
\newcommand{\defref}[1]{\hyperref[#1]{Definition \ref{#1}}}
\newcommand{\rmkref}[1]{\hyperref[#1]{Remark \ref{#1}}}
\newcommand{\qref}[1]{\hyperref[#1]{Question \ref{#1}}}
\newcommand{\secref}[1]{\hyperref[#1]{\S\ref{#1}}}
\newcommand{\appref}[1]{\hyperref[#1]{Appendix \ref{#1}}}

\nc{\St}{\mathrm{st}}
\nc{\df}{\mathrm{df}}

\nc{\gbmod}{\textrm{-gmod-}}

\nc{\gmod}{\textrm{-gmod}}
\author{Joshua Ciappara and Geordie Williamson}
\date{April 2020}
\title{Lectures on the Geometry and Modular Representation Theory of Algebraic Groups}

\usepackage{fancyhdr}
\begin{document}
\maketitle

\begin{abstract}
These notes provide a concise introduction to the representation theory of reductive algebraic groups in positive characteristic, with an emphasis on Lusztig's character formula and geometric representation theory. They are based on the first author's notes from a lecture series delivered by the second author at the Simons Centre for Geometry and Physics in August 2019. We intend them to complement more detailed treatments.

  \end{abstract}

\tableofcontents 
\section*{Introduction}
\subsection{Group actions}
In mathematics, group actions abound; their study is rewarding but challenging. To make problems more tractable, an important approach is to \textit{linearise} actions and focus on the \textit{representations} that arise. Historically, the passage from groups to representations was a non-obvious step, arising first in the works of Dedekind, Frobenius, and Schur at the turn of the last century.\footnote{For a fascinating account of this history, we recommend \cite{curtis1999}.} Nowadays it pervades modern mathematics (e.g. the Langlands program) and theoretical physics (e.g. quantum mechanics and the standard model).

In the universe of all possible representations of a group, the ones we encounter by linearising are typically well behaved in context-dependent ways; we say these representations ``occur in nature'':
\begin{enumerate}
\item Any representation of a finite group occurs inside a representation obtained by linearising an action on a finite set; thus all representations of finite groups ``occur in nature''. Over the complex numbers Maschke's theorem, Schur's lemma, and character theory provide powerful tools for understanding the entire category of representations.
\item Lie group actions on smooth manifolds $M$ induce representations on $L^2(M,\mathbb{C})$ and, more generally, on the sections and cohomology spaces of equivariant vector bundles on $M$. Here it is the \textit{unitary representations} which are most prominent. The study of continuous representations of Lie groups is the natural setting for the powerful Plancherel theorems and abstract harmonic analysis.
\item The natural permutation of polynomial roots by a Galois group $\Gamma$ produces interesting representations after linearising (so-called \textit{Artin representations}). More generally, Galois group actions on \'etale and other arithmetic cohomology theories produce continuous representations (so-called \textit{Galois representations}) which are fundamental to modern number theory.
\end{enumerate}
These notes concern algebraic representations of algebraic groups. In algebraic geometry, the actions that occur in nature are the \textit{algebraic actions}; linearising leads to \textit{algebraic representations}.\footnote{One often finds the term \textit{rational representations} in the literature. We try to avoid this terminology here, as we find it often leads to confusion.} For example, an algebraic group $G$ acting on a variety then acts algebraically on its regular functions. More generally, $G$ acts algebraically on the sections and cohomology groups of equivariant vector bundles.

Among all algebraic groups, our main focus will be on the representation theory of \textit{reductive algebraic groups}. These are the analogues in algebraic geometry of compact Lie groups. Indeed, over an algebraically closed field of characteristic zero, the representation theory of reductive algebraic groups closely parallels the theory of continuous finite-dimensional representations of compact Lie groups: the categories involved are semi-simple, simple modules are classified by highest weight, and characters are given by Weyl's famous formula.

Over fields of characteristic $p$ the classification of simple modules is still by highest weight, but a deeper study of the categories of representations yields several surprises. First among these is the Frobenius endomorphism, which is a totally new phenomenon in characteristic $p$, and implies immediately that the categories of representations must behave differently to their characteristic zero cousins.

\subsection{Simple characters}
A basic question underlying these notes is the determination of the characters of simple modules. Understanding their characters is a powerful first step towards understanding their structure. Equally or perhaps more importantly, the pursuit of character formulas has motivated and been parallel to rich veins of mathematical development.

A beautiful instance of this was in the conjecture and proof of a character formula for simple highest weight modules over a complex semi-simple Lie algebra $\mathfrak{g}$. The resolution of the Kazhdan--Lusztig conjecture by Brylinski--Kashiwara \cite{BK} and Beilinson--Bernstein \cite{bb} in 1981 hinged on a deep statement relating $D$-modules on the flag variety of $G$ to representations of $\mathfrak{g}$. The geometric methods introduced in \cite{bb} were one of the starting points of what is now known as \textit{geometric representation theory}, and the localisation theorem remains a tool of fundamental importance and utility in this area.


    The analogous question over algebraically closed fields of positive characteristic has resisted solution for a longer period and demanded the adoption of totally different approaches. From the time it was posited \cite{lus} until very recently, the state of the art has been Lusztig's conjectural character formula for simple $G$-modules. Our main goal in these lectures will be to state and then examine this conjecture, particularly in terms of its connections to perverse sheaves and geometry. We conclude with a brief discussion of how the conjecture was found to be correct for large $p$, but also how the expected bounds were too optimistic. Moving along a fast route towards fundamental open questions in modular representation theory, we will encounter many of the objects, results, and ideas which underpin this discipline.
    
\subsection{Outline of contents} 
\begin{description}
\item[Lecture I] We introduce algebraic groups and their representations, as well as the Frobenius morphisms which give the characteristic $p$ story its flavour. 

\item[Lecture II] We narrow the lens to reductive groups $G$ and their root data, before making connections between the representation theory of $G$ and the geometry of the flag variety $G/B$ (for $B$ a Borel subgroup). 

\item[Lecture III] We explore two analogous character formula conjectures: one for semi-simple Lie algebras in characteristic 0, due to Kazhdan--Lusztig, and one for reductive groups in characteristic $p$, due to Lusztig. 

\item[Lecture IV] We state Lusztig's conjecture more explicitly, before explaining its relation to perverse sheaves on the affine Grassmannian via the Finkelberg--Mirkovi\'c  conjecture.

\item[Lecture V] We discuss the phenomenon of torsion explosion and its bearing on estimates for the characteristics $p$ for which Lusztig's conjecture is valid. To finish, we give an illustrative example in an easy case, as well as indications of how the theory of intersection forms can be applied to torsion computations in general. 
\end{description}

\subsection{Notation}
Throughout these notes, we fix an algebraically closed field $k$ of characteristic $p \ge 0$; our typical focus will be $p > 0$. Unadorned tensor products are taken over $k$. Unless otherwise noted, modules are \textit{left} modules.

\subsection{Acknowledgements}
We would like to thank all who took part in the summer school for interesting discussions and an inspiring week. We are grateful to the Simons Centre for hosting these lectures and for their hospitality.

\makeatletter
\let\savedchap\@makeschapterhead
\def\@makeschapterhead{\vspace*{-1cm}\savedchap}
\chapter*{Lecture I}
\let\@makeschapterhead\savedchap

\section{Algebraic groups}

We start by introducing algebraic groups and their duality with commutative Hopf algebras, using the functor of points formalism. We follow Jantzen \cite{jan}. Readers desiring to pursue this material is greater depth will certainly require further details on both the algebraic and geometric sides; for this we recommend \cite{har} and \cite{wat} in addition to \cite{jan}.

\subsection{Schemes as functors}
\begin{defi}
A \textit{$k$-functor} $\mathcal{X}$ is any (covariant) functor from the category of commutative, unital $k$-algebras to the category of sets:
$$\mathcal{X}: \text{$k$-Alg} \to \text{Set}.$$
Such $k$-functors form a category $k$-Fun with natural transformations as morphisms.
\end{defi}

When first learning algebraic geometry, we think of a $k$-variety $X$ as the subset of affine space $k^n$ defined by the vanishing of an ideal $I \subseteq k[x_1, \dots, x_n]$. Grothendieck taught us to widen this conception of a variety by considering the vanishing of $I$ over any base $k$-algebra $A$:
$$X(A) = \{ \text{$a \in A^n: f(a) = 0$ for all $f \in I$} \}.$$
In other words, $X(A)$ is the solutions in $A$ of the equations defining $X$. The association $A \mapsto X(A)$ extends to a $k$-functor $\mathcal{X}: \text{$k$-Alg} \to \text{Set}$ as follows: if $\varphi: A \to B$ is a $k$-algebra homomorphism, then the identity $$f(\varphi(a)) = \varphi(f(a)), \quad f \in I,$$ shows there is an induced mapping $X(A) \to X(B)$. In this way, $k$-varieties provide the most important examples of $k$-functors.

The bijection $X(A) \cong \text{Hom}_{\text{$k$-Alg}}(k[X],A)$ gives a coordinate-free (though perhaps less intuitive) construction of $\mathcal{X}$ from $X$, and it underlies the next definition.

\begin{defi}
Let $R$ be a $k$-algebra.
\begin{enumerate}
\item The \textit{spectrum} $\text{Spec}_k(R)$ is the representable $k$-functor $\text{Hom}(R, -)$. 
\item The category of \textit{affine $k$-schemes} is the full subcategory of the category of $k$-functors given by spectra of $k$-algebras $R$.
\end{enumerate}
If $k$ is understood, we can suppress it from notation and write simply $\text{Spec}(R)$.
\end{defi}
In this way, we obtain a contravariant functor $\text{Spec}_k: \text{$k$-Alg} \to \text{$k$-Fun}$, since an algebra homomorphism $\varphi: A \to B$ induces a natural transformation
$$\text{Spec}_k(B) \to \text{Spec}_k(A)$$
via pre-composition with $\varphi$. The anti-equivalence of \text{$k$-Alg} with the category of affine $k$-varieties now shows we have embedded the latter category inside of \text{$k$-Fun}. We henceforth drop the distinction in notation between $X$ and $\mathcal{X}$. 

\begin{defi}
Let $\mathbb{A}_k^n = \text{Spec}_k(k[x_1,\dots,x_n])$ be affine $n$-space over $k$. 
\begin{enumerate} 
\item If $X$ is a $k$-functor, then define
$$k[X] = \text{Hom}_{\text{$k$-Fun}}(X,\mathbb{A}_k^1),$$
the \textit{regular functions} on $X$. It is a $k$-algebra under pointwise addition and multiplication.
\item Say the affine $k$-scheme $X$ is \textit{algebraic} if $k[X]$ is of finite type over $k$ (i.e. finitely generated as a $k$-algebra), and \textit{reduced} if it contains no non-zero nilpotents.
\end{enumerate}
\end{defi}
This definition generalises the algebra of global functions on a $k$-variety. We should now define \textit{$k$-schemes} to be $k$-functors which are locally affine $k$-schemes in an appropriate sense. Since we will work directly with relatively few non-affine schemes, we omit the precise technical developments here and refer the reader to \cite[\textsection I.1]{jan}. 

\begin{exercise}
Consider the $k$-functor $F$ defined by the rule
    $$F(A) = \{ a \in A^{\mathbb{N}}: \text{$a_i = 0$ for all but finitely many $i \in \mathbb{N}$} \}.$$
Show that $F$ is not an affine $k$-scheme.
\end{exercise}

\subsection{Group schemes}
\begin{defi}
A \textit{$k$-group functor} is a functor $\text{$k$-Alg} \to \text{Grp}$. A \textit{$k$-group scheme} (resp. \textit{algebraic $k$-group}) is a $k$-group functor whose composite with the forgetful functor $\text{Grp} \to \text{Set}$ is an affine $k$-scheme (resp. algebraic affine $k$-scheme).
\end{defi}

Equivalently, $k$-group schemes are group objects in the category of affine $k$-schemes. From this viewpoint, it is straightforward to see that they correspond to commutative Hopf algebras in $k$-Alg under the aforementioned anti-equivalence:
\begin{equation} \label{anti}
\{ \text{$k$-group schemes} \} \cong \{ \text{commutative Hopf algebras} \}^\text{op}, \quad G \mapsto k[G].
\end{equation}
In some situations it is more convenient to specify an algebraic group by its Hopf algebra. For more discussion of this, see \cite[\textsection I.2.3--2.4]{jan}.

The following is an important source of $k$-group functors.

\begin{defi} \label{grpassoc}
Let $V$ be a $k$-vector space. The \textit{$k$-group functor $V_a$ associated to} $V$ is given by $V_a(A) = (V \otimes A,+)$.
\end{defi}

\begin{ex} \leavevmode
\begin{enumerate}
\item The \textit{additive group} $\mathbb{G}_a$ is defined on $k$-algebras by
    $$\mathbb{G}_a(A) = (A,+).$$
    In other words, $\mathbb{G}_a = k_a$ is the notation of Definition \ref{grpassoc}. We have $k[\mathbb{G}_a] = k[z]$, a polynomial ring in one variable, with $$\Delta(z) = 1 \otimes z + z \otimes 1, \quad \varepsilon(z) = 0, \quad S(z) = -z$$
    as comultiplication, counit, and antipode.
    
    \item The \textit{multiplicative group} $\mathbb{G}_m$ is defined by
    $$\mathbb{G}_m(A) = A^\times.$$
    Here $k[\mathbb{G}_m] = k[z,z^{-1}]$ and 
    $$\Delta(z) = z \otimes z, \quad \varepsilon(z) = 1, \quad S(z) = z^{-1}.$$
    A \textit{torus} over $k$ is any $k$-group isomorphic to an $n$-fold product $\mathbb{G}_m^n$.
    
    \item The \textit{$m$-th roots of unity} $\mu_m$ are a $k$-subgroup scheme of $\mathbb{G}_m$ defined by 
    $$\mu_m(A) = \{a \in A^\times: a^m = 1 \}.$$
    We have $k[\mu_m] = k[z]/(z^m - 1).$
    
    \item Let $M$ be a $k$-vector space and define $\text{GL}_M$ by
    $$\text{GL}_M(A) = \text{End}_A(M \otimes A)^\times.$$
    This is an affine $k$-scheme if and only if $M \cong k^n$ is finite dimensional, in which case it is an algebraic group with
    $$k[GL_M] = k[GL_n] = k[z_{ij}]_{1 \le i, j \le n}[(\text{det}(z_{ij})^{-1}],$$
    where we write $\text{GL}_n$ for $GL_{k^n}$. (Indeed, if $\{ m_i \}_{i \in I}$ is a basis of $M$, then there are regular coordinate functions $X_{ij} \in k[GL_n]$ for $i, j \in I$ whose non-vanishing sets would give an open cover of $GL_M(k)$ if it were the spectrum of some ring; by quasi-compactness, this forces $I$ to be finite.) Notice that $\text{GL}_1 = \mathbb{G}_m$.
    
    \item The upper triangular matrices with diagonal entries $1$ form a $k$-subgroup scheme $U_n \subseteq \text{GL}_n$.
\end{enumerate}
\end{ex}

\subsection{Base change}
Before moving on, let us note that we could have developed the above theory over an arbitrary commutative ring $R$, rather than the field $k$; this yields notions of $R$-schemes and $R$-group functors, etc. Then, given a ring map $f: R \to S$ and an $R$-scheme $X$, we can define its \textit{base change} $X_S$ from $R$ to $S$ by the formula $X_S(A) = X(A_R)$ for any $S$-algebra $A$, where $A_R$ means $A$ viewed as an $R$-algebra (with structure map $R \to S \to A$). We find that $X_S$ is an $S$-scheme and it fits into the following pullback square in the category of $R$-schemes:
\[
\begin{tikzcd}
X_S \arrow{d}{} \arrow{r}{} & X \arrow{d}{} \\
\text{Spec}(S) \arrow{r}{\text{Spec}(f)} & \text{Spec}(R)
\end{tikzcd}
\]
If $X$ is $R$-affine, then $X_S$ is $S$-affine with regular functions $S[X_S] = S \otimes_R R[X]$.

\begin{defi}
Let $Y$ be an $S$-scheme and $R$ a subring of $S$. We say $Y$ is \textit{defined over $R$} in case there is an $R$-scheme $X$ for which $X_S \cong Y$ as $S$-schemes.
\end{defi}

\section{Representations}
The main purpose of this section is to develop three equivalent viewpoints on what it means to \textit{represent} an algebraic group $G$ on $k$-vector spaces. This will parallel the classical dictionary between representations of a finite group and modules over its group ring.

\begin{defi}
A \textit{representation} of $G$ is a homomorphism of $k$-group functors $$G \to \text{GL}_V,$$
where $V$ is some $k$-vector space. 
\end{defi}

Suppose $G$ is reduced and $V \cong k^n$ is finite dimensional. A representation of $G$ on $V$ is equivalent to a group homomorphism
$$G(k) \to \text{GL}_n(k), \quad g \mapsto (z_{ij}(g)),$$
where the matrix coefficients are regular functions $z_{ij} \in k[G]$. This is an intuitive way to picture representations. 
\newpage

\begin{defi} \leavevmode
\begin{enumerate}
\item Let $G$ be an algebraic $k$-group and $V$ a $k$-vector space. A \textit{(left) $G$-module} structure on $V$ is an action of $G$ on the $k$-functor $V_a$, i.e. a natural transformation
$$G \times V_a \to V_a$$
such that the induced action of $G(A)$ on $V \otimes A$ is $A$-linear for each $A$.
\item A \textit{G-module homomorphism} from a $G$-module $V$ to a $G$-module $W$ is a $k$-linear map $f: V \to W$ such that
$$(f \otimes 1)(g \cdot (v \otimes a)) = g \cdot (f(v) \otimes a),$$
for all $g \in G$, $v \in V$, and $a \in A$; here dots denote the action maps of $G$ on $V_a$ and $W_a$.
\end{enumerate}
\end{defi}

In view of the anti-equivalence \eqref{anti}, modules for $G$ correspond to a certain type of ``dual'' representation object for the Hopf algebra $k[G]$. 

\begin{defi}
Let $H$ be a Hopf algebra over $k$. A \textit{(right) comodule} over $H$ is a $k$-module $V$ equipped with a $k$-linear map $\mu: V \to V \otimes H$, such that
$$(1_V \otimes \Delta) \circ \mu = (\mu \otimes 1_H) \circ \mu \quad \text{and} \quad (1 \otimes \varepsilon) \circ \mu = 1_V,$$
where $\Delta$ and $\varepsilon$ are the comultiplication and counit of $H$, respectively, and we identify $V \otimes k \cong V$.
\end{defi}

Now we are ready to assert the existence of a dictionary between representations, modules, and comodules.

\begin{prop} \label{equiv}
There are natural equivalences of categories:
$$\{ \text{representations of $G$} \} \cong \{ \text{left $G$-modules} \} \cong \{ \text{right $k[G]$-comodules} \}.$$
\end{prop}

\begin{remark} \label{ops}
In fact, all three categories are \textit{abelian tensor categories}, and the equivalences respect these structures. So, for instance, $G$-modules $M$ and $N$ can be used to construct new $G$-modules $M \oplus N$ and $M \otimes N$. Observe also that $$M^* = \text{Hom}_k(M,k)$$ is naturally a $G$-module, the \textit{dual} of $M$, and thus so is $M^* \otimes N = \text{Hom}_k(M,N).$
\end{remark}

\begin{exercise}
Prove Prop. \ref{equiv}, formulating the appropriate notion of a morphism in the first and third categories. Then verify the details of Remark \ref{ops}.
\end{exercise}

It is important to be fluent in moving between the different notions of representations. However, we will ordinarily think of them as $G$-modules, and hence use the notation $\text{Rep}(G)$ for the abelian category of finite-dimensional $G$-modules $V$. In practice, results on representations can sometimes be obtained most expediently via comodules; an example follows.

\begin{prop} \label{lfinite}
If $G$ is an algebraic group and $V$ is a $G$-module, then $V$ is locally finite: any finite-dimensional subspace of $V$ is contained in a finite-dimensional $G$-stable subspace of $V$. 
\end{prop}

\begin{proof}
View $V$ as a right $k[G]$-comodule with action map
$a: V \to V \otimes k[G]$, and suppose that for some fixed $v$ we have 
\begin{equation} \label{action}
a(v) = \sum_{i = 1}^r v_i \otimes f_i
\end{equation}
with respect to a fixed choice of basis $\{ f_i \}$ of $k[G]$. Then $g \cdot v = \sum f_i(g) v_i$ for all $g \in G$, implying $v \in W = \sum_{i=1}^r kv_i$. Since $(a \otimes 1) \circ a = (1 \otimes \Delta) \circ a$, we can apply $a \otimes 1$ to the right-hand side of \eqref{action} and expand it in two different ways:
$$\sum_i a(v_i) \otimes f_i = \sum_i \left ( \sum_k v_k^i \otimes f_k \right ) \otimes f_i = \sum_i v_i \otimes \Delta(f_i).$$
If $\varepsilon$ denotes evaluation at 1 (the counit of $k[G]$) and $\rho_g: k[G] \to k[G]$ is the action of $g \in G$ in the regular representation of $G$ on $k[G]$, then we can consider $\varepsilon_g = \varepsilon \circ \rho_{g^{-1}}$, i.e. evaluation at $g$. Now apply $1 \otimes \varepsilon_g \otimes f_i^*$ to the previous equation and simplify, where $f_i^* \in k[G]^*$ is defined by $f_i^*(f_j) = 0$:
$$g \cdot v_i = \sum_j f_i^*(\eta_j) v_i,$$
for $\eta_j = ((\varepsilon_g \otimes 1) \circ \Delta)(f_j) \in k[G]$. This shows $W$ is a finite-dimensional $G$-stable subspace of $V$ containing $v$. 
\end{proof}

\begin{cor}
Simple representations of $G$ are finite dimensional.
\end{cor}

\begin{ex} \leavevmode
\label{repex}
\begin{enumerate}
    \item For any algebraic group $G$ and any vector space $V$, we have the \textit{trivial representation} $V_\text{triv}$ on $V$, via the trivial group homomorphism $G \to GL_V$.
    
    \item The prototypical representation is the \textit{regular representation} $k[G],$ obtained by viewing $k[G]$ as a comodule over itself. The comodule action map $$a: V \to V \otimes k[G]$$ of a comodule $V$ can be interpreted as an embedding
    $$V \hookrightarrow V_{\text{triv}} \otimes k[G].$$
    This shows that any representation embeds within a direct sum of regular representations, and also that any irreducible representation is a submodule of $k[G].$
    
    \item There is a decomposition of $k[\mathbb{G}_m]$ into one-dimensional $\mathbb{G}_m$-stable subspaces, $$k[\mathbb{G}_m] = \bigoplus_m k z^m,$$ with $a \cdot z^m = a^{-m} z^m$. These components are precisely the simple $\mathbb{G}_m$-modules, and that any representation of $\mathbb{G}_m$ is semi-simple.
    
    \item In the regular representation of $G = \mathbb{G}_a$ on $k[z]$, there is an increasing filtration by indecomposable submodules 
    $$V_i = \{ f \in k[z]: \text{deg f} \le i \}.$$
    Indeed, $\lambda \cdot z = z + \lambda$, for $\lambda \in \mathbb{G}_a$. 
    
    Suppose now that $p > 0$. Then $k \oplus kz^p$ is a $G$-stable subspace we would not see in characteristic zero. To understand why it arises is one of the goals of the next section.
\end{enumerate}
\end{ex}

\begin{exercise} \label{torusact}
Let $T = \mathbb{G}_m^r$ be a torus. Show that there is a canonical equivalence of categories
    $$\text{Rep}(T) \cong \{ \text{$X(T)$-graded $k$-modules} \},$$
where $X(T) = \Hom(T,\mathbb{G}_m)$ is the \textit{character lattice} of $T$. (Hint: this becomes very transparent in the language of comodules.)
\end{exercise}

\section{Frobenius kernels} \label{sec:frob}
\subsection{Constructions and definitions} In this section we assume $p > 0$. Given a $k$-algebra $A$ and $m \in \mathbb{Z}$, we can define a new $k$-algebra structure on the ring $A$ by
$$c \cdot a = c^{p^{-m}}a,$$
for $c \in k$ and $a \in A$; denote the resulting $k$-algebra by $A^{(m)}$. The $p$-th power map $A \to A$, $x \mapsto x^p,$
which is normally only a homomorphism of $\mathbb{F}_p$-algebras, can now be viewed as a $k$-algebra homomorphism $\sigma_A: A \to A^{(-1)}$.

Let us extrapolate this construction into geometry. Given a $k$-scheme $X$, we can form the base change $X^{(1)}$ of $X$ along $\sigma_k: k \to k^{(-1)}$. Since $k^{(-1)}$ agrees with $k$ as a ring, $X^{(1)}$ is a new $k$-scheme, fitting into the following pullback diagram:
\[
\begin{tikzcd}
X^{(1)} \arrow{d}{} \arrow{r}{} & X \arrow{d}{} \\
\text{Spec}(k^{(-1)}) \arrow{r}{\text{Spec $\sigma_k$}} & \text{Spec}(k)
\end{tikzcd}
\]
We refer to $X^{(1)}$ as the \textit{Frobenius twist} of $X$. If $X$ is affine, then $$k[X^{(1)}] = k^{(-1)} \otimes_k k[X];$$ 
more generally, $X^{(1)}$ is given as a functor by 
$X^{(1)}(A) = X(A^{(-1)}).$ In particular, the maps $X(\sigma_A): X(A) \to X(A^{(-1)})$ give rise to a \textit{Frobenius morphism} 
$$\text{Fr}: X \to X^{(1)}.$$
Its composite with the universal map $X^{(1)} \to X$ is known as the \textit{absolute Frobenius morphism} $\text{Fr}_{\text{abs}}: X \to X$.
\begin{center}
\begin{tikzcd}
X
\arrow[drr, bend left, "\text{Fr}_{\text{abs}}"]
\arrow[dr, "{\text{Fr}}" description] & & \\
& X^{(1)} \arrow[r, ""] \arrow[d, ""]
& X \arrow[d] \\
& \text{Spec}(k^{(-1)}) \arrow[r, "\text{Spec $ \sigma_k$}"]
& \text{Spec} (k).
\end{tikzcd}
\end{center}

\begin{exercise} \label{froiso}
If $X$ is defined over $\mathbb{F}_p$, then $X^{(1)} \cong X$.
\end{exercise}

\begin{exercise}
Suppose $X$ is a closed subvariety of $\mathbb{A}^n$ defined by $$f_1, \dots, f_m \in k[\mathbb{A}^n].$$ Establish defining equations for $X^{(1)}$ as a subvariety of $\mathbb{A}^n$, and explicitly describe the morphisms from the previous diagram in this setting.
\end{exercise}

Iterating the construction of Fr, we get a chain of morphisms
$$X \to X^{(1)} \to X^{(2)} \to \cdots;$$
the composite $X \to X^{(n)}$ is denoted $\text{Fr}^n$. Importantly, if $G$ is a $k$-group scheme, then so are its Frobenius twists and $\text{Fr}^n$ is a homomorphism of $k$-group schemes. Pulling back along these homomorphisms yields Frobenius twist functors 
$$\text{Rep}(G^{(n)}) \to \text{Rep}(G), \quad V \mapsto V^{\text{Fr}^n}.$$

\begin{Eg}
Identifying $\mathbb{G}_a^{(1)} \cong \mathbb{G}_a$ (as in Exercise \ref{froiso}), we have $V_1^{\text{Fr}} \cong k \oplus kz^p$, in the notation of Example \ref{repex}(4).
\end{Eg} 

\begin{defi}
The \textit{$n$-th Frobenius kernel} of a $k$-group scheme $G$ is its subgroup scheme $$G_n = \text{ker $\text{Fr}^n$} \le G.$$
\end{defi}

\begin{exercise} \label{addker} \leavevmode
\begin{enumerate}
    \item Verify that $k[\mathbb{G}_{a,n}] = k[z]/(z^{p^n})$.
    \item Show that a finite-dimensional representation of $\mathbb{G}_{a,n}$ is equivalent to the data $(V, \phi_1, \dots, \phi_n)$, where $V$ is finite dimensional and the $\phi_i$ are commuting operators on $V$ with $\phi_i^p = 0$.
\end{enumerate}
\end{exercise}

\subsection{Representations of Frobenius kernels}
To conclude this lecture, we indicate the theoretical significance of Frobenius kernels. From here onward, we will need to draw on background from Lie theory; \cite{hum1} and \cite[\textsection III]{hum} are good introductory references. Let $G$ be an $k$-group scheme and let
$$\mathfrak{g} = T_1 G = \text{Der}(k[G],k)$$
denote the Lie algebra of $G$, whose underlying vector space is the tangent space of $G$ at the identity. We can identify $X \in \text{Der}(k[G],k)$ with left-invariant $k$-derivations $D$ from $k[G]$ to itself, i.e. those for which the following square commutes:

\[
\begin{tikzcd}
k[G] \arrow{d}{D} \arrow{r}{\Delta} & k[G] \otimes k[G] \arrow{d}{1 \otimes D} \\
k[G] \arrow{r}{\Delta} & k[G] \otimes k[G]
\end{tikzcd}
\]
The bracket on $\mathfrak{g}$ is then the commutator of derivations, and furthermore we can see $\mathfrak{g}$ is a $p$-Lie algebra with $X^{[p]} = X^p$. 

Regardless of the characteristic of $k$, there is a functor of ``differentiation'',
$$D: \text{Rep}(G) \to \text{Rep}(\mathfrak{g}),\footnote{By $\text{Rep}(\mathfrak{g})$ we denote the category of finite-dimensional representations of the Lie algebra $\mathfrak{g}$.}$$
obtained according to the following recipe: Given a $k[G]$-comodule $V$ with action map $a: V \to V \otimes k[G]$, we define
$$X \cdot v = (1 \otimes X)(a(v)), \quad X \in \mathfrak{g} = \text{Der}(k[G],k), \quad v \in V.$$
We then have the following useful proposition.

\begin{prop} \label{ker}
Assume $G$ is connected. 
\begin{enumerate}
    \item In characteristic 0, $D$ is fully faithful.
    \item In characteristic $p$, $D$ induces an equivalence 
    $$\text{Rep}(G_1) \cong \text{$u(\mathfrak{g})$-mod},$$
    where $u(\mathfrak{g}) = U(\mathfrak{g})/(X^p - X^{[p]}).$
\end{enumerate}
\end{prop}

\begin{remark}
Let us outline the proof of Prop. \ref{ker}(2). If $A$ is a finite-dimensional Hopf algebra over $k$, then the dual $k$-vector space $A^*$ is naturally a Hopf algebra: all the structure maps are transposes of structure maps of $A$. For instance, the multiplication of $A^*$ is the image of the comultiplication of $A$ under the isomorphism
$$\text{Hom}(A, A \otimes A) \cong \text{Hom}(A^* \otimes A^*, A^*).$$
The correspondence $A \leftrightarrow A^*$ hence defines a self-duality on the category of finite-dimensional Hopf algebras over $k$, with the additional property that
\begin{equation} \label{dual1}
\{ \text{left $A$-modules} \} \cong \{ \text{right $A^*$-comodules} \}.
\end{equation}
Now, it can be shown that 
\begin{equation} \label{dual2}
    k[G_1] \cong u(\mathfrak{g})^*,
\end{equation}
while we have seen in Proposition \ref{equiv} that for every algebraic $k$-group $H$ there is an equivalence of categories
\begin{equation} \label{dual3}
\text{Rep}(H) \cong \text{$k[H]$-comod}.
\end{equation}
The equivalence $\text{Rep}(G_1) \cong \text{$u(\mathfrak{g})$-mod}$ is then a corollary of \eqref{dual1}, \eqref{dual2}, and \eqref{dual3}.  
\end{remark}

The intuition behind Proposition \ref{ker} is that the underlying field's characteristic has a strong bearing on the size of an algebraic group's subgroups, and so on how much of the group's representation theory is ``seen around the identity" by $\mathfrak{g}$. In characteristic zero, $G$ has ``no small subgroups", while in characteristic $p > 0$, $G$ has ``many small subgroups'' (particularly the Frobenius kernels). To be precise, the property of \textit{having small subgroups} means that every neighbourhood $U$ of the identity in $G$ contains a subgroup $H \le G$.
\begin{figure}[h]
\centering
\begin{tikzpicture}[scale=0.6, every node/.style={scale=0.75}]
\coordinate (O) at (0,0);
\draw (0,2.5) node {$G$};
\draw (0,0) node[above] {$1$} node {$\bullet$};
\draw[thick] (O) circle (3.25);
\draw[blue] (-3,-2) to[out=90,in=180] (0,0);
\draw[blue] (0,0) to[bend right] (3,2);
\draw[blue] (0,0) to[bend left] (-2,3);
\draw[blue] (0,0) to[bend left] (2,-3);

\draw[decoration={text along path,reverse path,text align={align=center},text={}},decorate] (0.6,0) arc (0:180:0.6);
\draw[decoration={text along path,reverse path,text align={align=center},text={}},decorate] (1.6,0) arc (0:180:1.6);
\draw[decoration={text along path,reverse path,text align={align=center},text={}},decorate] (2.6,0) arc (0:180:2.6);

\begin{scope}[xshift=8cm]
\coordinate (O) at (0,0);
\draw[thick] (O) circle (3.25);
\draw (0,2.5) node {$G$};
\draw[purple] (O) circle (1.9);
\draw[purple] (O) circle (1.2);
\draw (0,0) node[above] {$1$} node {$\bullet$};
\draw (O) circle (3.25);
\draw[blue] (-3,-2) to[out=90,in=180] (0,0);
\draw[blue] (0,0) to[bend right] (3,2);
\draw[blue] (0,0) to[bend left] (-2,3);
\draw[blue] (0,0) to[bend left] (2,-3);

\draw (0,-1) node[above] {$G_1$};
\draw (0,-1.8) node[above] {$G_2$};
\draw (0,-2.6) node[above] {$\vdots$};
\draw[decoration={text along path,reverse path,text align={align=center},text={}},decorate] (2.9,0) arc (0:180:2.9);
\end{scope}
\end{tikzpicture}
\caption{Left: the characteristic zero picture, with ``no small subgroups''. Right: the modular picture, with ``many small subgroups''. Closed subgroups, indicated with curved lines, are generally abundant in both contexts.} 
\end{figure}
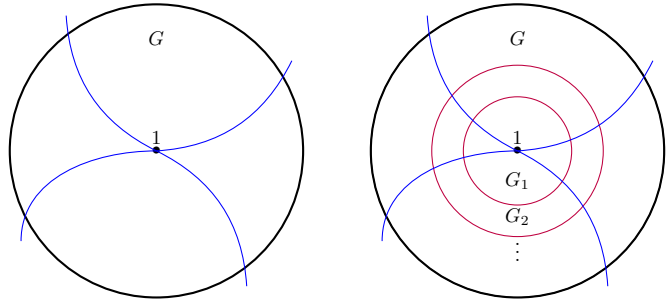

We also have a result that $\text{Rep}(G) \hookrightarrow \text{2-}\varprojlim \text{Rep}(G_m);$ here we refer to a 2-limit of categories, viewing them as objects in some appropriate 2-category. Practically speaking, this means that for any $V, V' \in \text{Rep}(G)$, there is $n \ge 1$ such that
$$\text{Hom}_G(V,V') = \text{Hom}_{G_n}(V,V').$$
In this sense, the family of Frobenius kernels of $G$ controls the representation theory of $G$.

\begin{exercise}
Show that in characteristic $p$,
$$\text{Rep}(\mathbb{G}_a) \cong \{ (V, \phi_n)_{n \ge 1}: \text{$V$ a $k$-vector space, $\phi_i \in \text{End}_k(V)$ with $\phi^p = 0$} \}.$$
This description is visibly the direct limit of the description in Exercise \ref{addker}. On the other hand, show that in characteristic zero the right-hand side should instead consist of pairs $(V,\phi)$ with $\phi: V \to V$ nilpotent. 
\end{exercise}
\newpage

\makeatletter
\let\savedchap\@makeschapterhead
\def\@makeschapterhead{\vspace*{-1cm}\savedchap}
\chapter*{Lecture II}
\let\@makeschapterhead\savedchap

\section{Reductive groups and root data}
In this lecture, we will restrict our attention to an important and well-studied class of algebraic groups. We approach this as directly as possible, recommending sources such as \cite{hum} or \cite{spring} for much more detail.

\begin{defi}
An algebraic group is \textit{unipotent} if it is isomorphic to a closed subgroup scheme of $U_n$.
\end{defi}

\begin{defi} Let $G$ be a (topologically) connected algebraic group over the algebraically closed field $k$.
\begin{enumerate}
\item $G$ is \textit{semi-simple} if the only smooth connected solvable normal subgroup of $G$ is trivial.
\item $G$ is \textit{reductive} if the only smooth connected unipotent normal subgroup of $G$ is trivial.
\end{enumerate}
\end{defi}
It can be shown that any unipotent group over $k = \overline{k}$ admits a composition series in which each quotient is isomorphic to $\mathbb{G}_a$. In particular, all unipotent groups are solvable, so all semi-simple groups are reductive. 

\begin{Eg}
The archetypal reductive group is $\text{GL}_n$. It contains many tori, which are also reductive. A maximal such torus, i.e. one contained in no other, is the subgroup of diagonal matrices $D_n \cong \mathbb{G}_m^n$.
\end{Eg}

Let $G$ be a reductive, connected algebraic group over $k$. The group's action on itself by conjugation defines a homomorphism of $k$-group functors $G \to \text{Aut}(G)$, and automorphisms of $G$ can be differentiated to elements of $\text{Aut}(\mathfrak{g})$. Thus we obtain the adjoint action of $G$ on $\mathfrak{g}.$ 

Recall from Exercise \ref{torusact} that a representation of a torus $T$ on a vector space $V$ is equivalent to a grading of $V$ by $X(T) = \Hom(T, \mathbb{G}_m)$. With respect to the adjoint action of a maximal torus $T \subseteq G$ on $V = \mathfrak{g}$, there is a decomposition
$$\mathfrak{g} = \text{Lie}(G) = \text{Lie}(T) \oplus \bigoplus_{\alpha \in R} \mathfrak{g}_\alpha.$$
Here $R \subseteq \mathfrak{X} = X(T)$ are the \textit{roots} relative to $T$, and $\mathfrak{g}_\alpha$ is the subspace upon which $T$ acts with character $\alpha$; by definition, $\mathfrak{g}_\alpha \ne 0$ for $\alpha \in R$. Pulling back through $\alpha: T \to \mathbb{G}_m$, the natural action of $\mathbb{G}_m$ on $\mathbb{G}_a$ by multiplication yields an action of $T$ on $\mathbb{G}_a$. Up to scalar, there is a unique \textit{root homomorphism} 
$x_\alpha: \mathbb{G}_a \to G$ which intertwines the actions of $T$ and induces an isomorphism
$$dx_\alpha: \text{Lie}(\mathbb{G}_a) \cong \mathfrak{g}_\alpha;$$
we denote its image subgroup by $U_\alpha$. After normalising $x_\alpha$ and $x_{-\alpha}$ suitably, we can construct $\varphi_\alpha: \text{SL}_2 \to G$ such that
$$\varphi_\alpha \begin{pmatrix}
    1 & a \\
    0 & 1
  \end{pmatrix} = x_\alpha(a) \quad \text{and} \quad \varphi_\alpha \begin{pmatrix}
    1 & 0 \\
    a & 1
  \end{pmatrix} = x_{-\alpha}(a).$$
Then, we get $\alpha^\vee \in \mathfrak{X}^\vee = Y(T) = \text{Hom}(\mathbb{G}_m,T)$ by defining 
$$\alpha^\vee(\lambda) = \varphi_\alpha \begin{pmatrix}
    \lambda & 0 \\
    0 & \lambda^{-1}
  \end{pmatrix},$$
and we write $R^\vee = \{ \alpha^\vee: \alpha \in R \} \subseteq \mathfrak{X}^\vee$. 

\begin{defi} \leavevmode
\begin{enumerate}
    \item A \textit{root datum} consists of a quadruple $(R \subseteq \mathfrak{X},R^\vee \subseteq \mathfrak{X}^\vee)$, along with a map $R \to R^\vee, \alpha \mapsto \alpha^\vee,$ satisfying the following conditions:
\begin{itemize}
    \item $\mathfrak{X}$, $\mathfrak{X}^\vee$ are free abelian groups of finite rank, equipped with a perfect pairing $\langle - , - \rangle: \mathfrak{X} \times \mathfrak{X}^\vee \to \mathbb{Z}$.
    \item $R$ and $R^\vee$ are finite and $\alpha \mapsto \alpha^\vee$ is bijective.
    \item For all $\alpha \in R$, we have $\langle \alpha, \alpha^\vee \rangle = 2$, and the function $s_\alpha: \mathfrak{X} \to \mathfrak{X}$ defined by
    $$s_\alpha(x) =  x - \langle x, \alpha^\vee \rangle \alpha$$
    permutes $R$ and induces an action on $\mathfrak{X}^\vee$ which restricts to a permutation of $R^\vee$.
\end{itemize}
Members of $R$ (resp. $R^\vee$) are called \textit{roots} (resp. \textit{coroots}). 
\item A \textit{morphism} of root data
$$(R \subseteq \mathfrak{X},R^\vee \subseteq \mathfrak{X}^\vee) \to (R_0 \subseteq \mathfrak{X}_0,R_0^\vee \subseteq \mathfrak{X}_0^\vee)$$
consists of an abelian group homomorphism $\phi: \mathfrak{X}_0 \to \mathfrak{X}$ which induces a bijection $R_0 \to R$, such that the \textit{transpose} $\phi^\text{tr}: \mathfrak{X}^\vee \to \mathfrak{X}_0^\vee$ induces a bijection $R^\vee \to R_0^\vee$.\footnote{The transpose $\phi^{\text{tr}}$ is uniquely determined by the requirement that $\langle \phi(x_0), x \rangle = \langle x_0, \phi^\text{tr}(x) \rangle_0$ for all $x \in \mathfrak{X}_0$, $x \in \mathfrak{X}$.} This gives meaning to the notion of an \textit{isomorphism} of root data: namely, $\phi$ is required to be a bijection.
\end{enumerate}
\end{defi}

We explained above how to construct a root datum from a reductive algebraic group $G$. It turns out this defines a bijection on isomorphism classes (see \cite[\textsection 1.3]{geck}).

\begin{thm}[Chevalley]
There is a one-to-one correspondence,
$$\{ \text{reductive algebraic groups over $k$} \}/\cong \quad \leftrightarrow \quad \{ \text{root data} \}/\cong.$$
\end{thm}

\begin{remark}
\leavevmode
\begin{enumerate}
    \item The bijection in the theorem is independent of $k$ (but crucially depends on the property of being algebraically closed). It turns out that for any root datum there exists a corresponding \textit{Chevalley group scheme} over $\mathbb{Z}$, whose base change to $k$ gives the corresponding reductive group over $k$.
    \item Interchanging $R \leftrightarrow R^\vee$ and $\mathfrak{X} \leftrightarrow \mathfrak{X}^\vee$ defines an obvious involution on the set of root data. On the other side of the bijection, this is a deep operation $G \leftrightarrow G^\vee$ on algebraic groups known as the \textit{Langlands dual}.
\end{enumerate}
\end{remark}
A root datum $(R \subseteq \mathfrak{X}, R^\vee \subseteq \mathfrak{X}^\vee)$ yields a finite Weyl group $$W_\text{f} = \langle s_\alpha : \alpha \in R \rangle \subseteq \text{Aut}_{\mathbb{Z}}(\mathfrak{X})$$
and also an \textit{abstract root system} within the subspace $V$ spanned by $R$ in the Euclidean space $\mathfrak{X}_{\mathbb{R}} = \mathfrak{X} \otimes_\mathbb{Z} \mathbb{R}$. In particular, there exists a choice of \textit{simple roots} $\Sigma \subseteq R$, which is a basis for $V$ such that any element of $R$ is a non-negative or non-positive integral linear combination from $\Sigma$. Then we obtain positive roots 
$$R_+ = \{ \alpha \in R: \text{$\alpha = \sum_{\sigma \in \Sigma} c_\sigma \sigma$ for $c_\sigma \in \mathbb{Z}_{\ge 0}$} \},$$
and simple reflections $S_{\text{f}} = \{ s_\alpha: \alpha \in \Sigma \}.$
Assume from now on that we are working with the root datum corresponding to a reductive group $G$, and fix choices $$\Sigma \subseteq R_+ \subseteq R$$ of positive (simple) roots. Corresponding to the choice of $R_+$ is a \textit{Borel subgroup} $T \subseteq B^+ = TU^+ \subseteq G$, where $U^+$ is the subgroup of $G$ generated by the $U_\alpha$ for $\alpha \in R_+$. For more detail on these objects, see \cite{bou} or \cite[\textsection II.1]{jan}.

\begin{ex}
\leavevmode
\begin{enumerate}
    \item Take $G = \text{GL}_n$ with maximal torus $T = D_n$. Then
    $$\mathfrak{X} = \bigoplus_i \mathbb{Z} \varepsilon_i, \quad \mathfrak{X}^\vee = \bigoplus_i \mathbb{Z} \varepsilon_i^\vee,$$
    where $\varepsilon_i(\text{diag}(\lambda_1, \dots, \lambda_n)) = \lambda_i$. The roots are
    $$R = \{ \varepsilon_i - \varepsilon_j: i \ne j \},$$
    and if we choose $R_+ = \{ \varepsilon_i - \varepsilon_j: i < j \}$ then $\Sigma = \{ \varepsilon_i - \varepsilon_{i+1} \}$ and $B^+$ is the set of upper triangular matrices.
    
    \item If $G = \text{SL}_n \le \text{GL}_n$ is the subgroup of matrices with determinant 1, then it contains a maximal torus $T \cong D_{n-1}$ consisting of the diagonal matrices with non-zero entries whose product is 1. We get
    $$\mathfrak{X} = \left ( \bigoplus_{i=1}^n \mathbb{Z} \varepsilon_i \right )/(\varepsilon_1 + \cdots + \varepsilon_n) \cong \bigoplus_{i=1}^{n-1} \mathbb{Z} (\varepsilon_i - \varepsilon_{i+1}),$$
    where by abuse of notation we conflate $\varepsilon_i - \varepsilon_{i+1}$ with its image in the indicated quotient. On the other hand, $\mathfrak{X}^\vee$ naturally identifies with the subgroup of $\oplus_i \mathbb{Z} \varepsilon_i^*$ whose coefficients in $\{ \varepsilon_i^* \}$ sum to zero. Then $$R = \{ \alpha_{ij} = \varepsilon_i - \varepsilon_j: i \ne j \}$$ and we can choose $\Sigma = \{ \varepsilon_i - \varepsilon_{i+1} \},$ for which $B^+$ is again the set of upper triangular matrices in $\text{SL}_2$. Up to scalar,
    $$x_{\alpha_{ij}}(\lambda) = I_n + \lambda e_{ij},$$
    where $e_{ij}$ denotes the matrix with 1 in position $(i,j)$ and zeroes elsewhere, so we can take $$\varphi_{\alpha_{ij}} \begin{pmatrix}
    a & b \\
    c & d
  \end{pmatrix} = a e_{ii} + b e_{ij} + c e_{ji} + de_{jj}$$
  for $i < j$. Thus we see $\alpha_{ij}^\vee = \varepsilon_i^\ast - \varepsilon_j^\ast$.
    \item Let $G = \text{PGL}_n$, the quotient of $\text{GL}_n$ by its centre. If $D_n = T \le \text{GL}_n$ is chosen as a maximal torus, then $q(T)$ is a maximal torus in $G$, where $q: \text{GL}_n \to G$ is the defining quotient map. We obtain that
    $$\mathfrak{X} = X(q(T)) = \left \{ \sum_{i=1}^n a_i \varepsilon_i: \sum a_i = 0 \right \} \subseteq X(T).$$
    The cocharacter lattice $\mathfrak{X}^\vee$ is isomorphic to $\left ( \oplus_i \mathbb{Z} \varepsilon_i^* \right )/(\varepsilon_1^* + \cdots + \varepsilon_n^*)$, where the image of $\varepsilon_i^*$ corresponds to the cocharacter $\lambda \mapsto I + (\lambda - 1)e_{ii}$. After determining roots and coroots, it becomes clear from our descriptions that $\text{PGL}_n$ is the Langlands dual of $\text{SL}_n$.
\end{enumerate}
\end{ex}

\begin{exercise}
Calculate the root data of $\text{Spec}_{2n}$, $\text{SO}_{2n}$, and $\text{SO}_{2n+1}$. Identify the Langlands dual in each case.
\end{exercise}

\section{Flag varieties}
\subsection{Geometric realisations of simple modules}
We have seen that every representation of $G$ embeds into a direct sum of copies of $k[G]$, or in other words, is ``seen by $k[G]$''. However, for reductive groups $G$, much can be gleaned by studying the flag variety $G/B^+$. In particular, we will see that simple representations of $G$ arise in spaces of global sections of sheaves on $G/B^+$.

\begin{ex}
\leavevmode
\begin{enumerate}
\item For $G = \text{SL}_2$ we have $G/B^+ = \mathbb{P}^1$. To see this, notice that we can identify $\mathbb{P}^1$ with the set $L$ of lines $0 \subseteq \ell \subseteq V = k^2$. There is an obvious transitive action of $G$ on $L$, under which the unique line $\ell \in L$ containing $e_1 = (1,0)$ has stabiliser $B^+$. Hence, the action map yields the stated isomorphism.
\item For entirely similar reasons, $G = \text{GL}_n$ is such that 
$$G/B^+ = \{ 0 \subseteq V_1 \subseteq \dots \subseteq V_n = k^n: \text{$V_i$ is an $i$-dimensional subspace} \}.$$
Indeed, let $\mathscr{F}$ denote the right-hand side, and let $F_0 \in \mathscr{F}$ be the \textit{standard flag} corresponding to an ordered basis of $k^n$ relative to which $B^+$ consists of upper triangular matrices; that is, $F_0$ is the flag with $V_i$ spanned by the standard basis vectors $e_1, \cdots, e_i$. Now the orbit map
$$G \to \mathscr{F}, \quad g \mapsto g \cdot F_0$$
induces a morphism $G/B^+ \to \mathscr{F}$ by the definition of quotient varieties; this turns out to be an isomorphism.
\end{enumerate}
\end{ex}

\begin{defi} \label{equidef}
Suppose $G$ acts on a $k$-scheme $X$ through $\sigma: G \times X \to X$. A \textit{G-equivariant sheaf} $\mathcal{F}$ on $X$ is a sheaf of $\mathcal{O}_X$-modules together with an isomorphism of $\mathcal{O}_{G \times X}$-modules 
$$\phi: \sigma^* \mathcal{F} \to p_2^* \mathcal{F}$$
which satisfies the cocycle condition
$$p_{23}^* \phi \circ (1_G \times \sigma)^* \phi = (m \times 1_X)^* \phi.$$
Here we refer to the obvious projections $p_{23}: G \times G \times X \to G \times X$, $p_2: G \times X \to X$, and multiplication $m: G \times G \to G$.
\end{defi}

\begin{remark}
On the stalk level, the first of these conditions ensures that $\mathcal{F}_{gx} \cong \mathcal{F}_x$ for all $x \in X$, and the second that the isomorphism $\mathcal{F}_{ghx} \cong \mathcal{F}_x$ coincides with $\mathcal{F}_{ghx} \cong \mathcal{F}_{hx} \cong \mathcal{F}_x$.
\end{remark}

Given a $G$-equivariant sheaf $\mathcal{F}$ on $X$ with $G$-action $\sigma$, the space of global sections $\Gamma(X, \mathcal{F})$ admits a natural $G$-module structure. While we will not need it directly, let us record the formula defining this action. If $g \in G$ and $w \in \Gamma(X,\mathcal{F})$, then
$$g \cdot w = (\phi_{G \times X} \circ \sigma_X^{\#})(w)(g^{-1}),$$
where $\phi$ is the $\mathcal{O}_{G \times X}$-module isomorphism required by Definition \ref{equidef} and $\mathcal{\sigma}_X^{\#}$ is the map on global sections $\Gamma(X, \mathcal{F}) \to \Gamma(G \times X, \sigma^* \mathcal{F})$ associated with $\sigma$.


Suppose that $V$ is a simple $G$-module. Then $G$ acts on $\mathbb{P}(V^*)$ and $\mathcal{O}(1)$ is an equivariant line bundle for the action. In particular, we recover the representation $V$ from the action on global sections,
$$\Gamma(\mathbb{P}(V^*),\mathcal{O}(1)) = V.$$
We now want to transfer this realisation to the flag variety. Either of the following facts may be adduced to prove that $B^+$ has a fixed point in its action on $\mathbb{P}(V^*)$.
\begin{thm}[Borel] 
If $H$ is a connected, solvable algebraic group acting through regular functions on a non-empty complete variety $W$ over an algebraically closed field, then there is a fixed point of $H$ on $W$.
\end{thm}
\begin{prop} \label{fixed}
Suppose $U$ is a unipotent group and $M$ is a non-zero $U$-module. Then $M^U \neq 0$. In particular, the trivial representation is a $U$-submodule of $M$.
\end{prop}
\begin{exercise}
Prove Proposition \ref{fixed}.
\end{exercise}
Acting on a $B^+$-fixed point yields a morphism $f: G/B^+ \to \mathbb{P}(V^*)$, which then fits into a diagram
\[
\begin{tikzcd}
f^* \mathcal{O}(1) \arrow{d}{} \arrow{r}{} & \mathcal{O}(1) \arrow{d}{} \\
G/B^+ \arrow{r}{f} & \mathbb{P}(V^*).
\end{tikzcd}
\]
On global sections we have a non-zero map $V \to \Gamma(G/B^+,f^* \mathcal{O}(1))$. Since $V$ is simple, this map is injective. Hence we can conclude that simple representations of $G$ occur in global sections of line bundles on the flag variety.

\begin{ex}
Let $G = \text{SL}_2$. On $\mathbb{P}^1 = G/B^+$, the line bundles $\mathcal{O}(n)$ have a unique equivariant structure. Recall that we can identify
$$\nabla_n = \Gamma(\mathbb{P}^1,\mathcal{O}(n)) = k[x,y]_{\text{deg $n$}} = ky^n \oplus ky^{n-1} x \oplus \cdots \oplus kx^n$$
if $n \ge 0$, and is zero otherwise. 

If $p = 0$, then the $\nabla_n$ are exactly the simple $\text{SL}_2$-modules; this follows (for example) from Lie algebra considerations and leads, for example, to the theory of spherical harmonics.

If $p > 0$, then $\nabla_n$ is simple for $0 \le n < p$, but $\nabla_p$ is not. Indeed, there is a $G$-submodule
$$L_p = kx^p \oplus ky^p \subseteq \nabla_p,$$
which is the Frobenius twist of $\nabla_1$, the natural representation of $\text{SL}_2$ on $k^2$. This is clear from the formula for the action of an arbitrary matrix:
$$\begin{pmatrix}
a & b \\
c & d 
\end{pmatrix} \cdot x^p = a^p x^p + c^p y^p; \quad \begin{pmatrix}
a & b \\
c & d 
\end{pmatrix} \cdot y^p = c^p x^p + d^p y^p.$$
In general, $L_n = G \cdot x^n \subseteq \nabla_n$ is simple; hence the simple modules are indexed by the same set as in the characteristic zero case, although their dimensions are different in general.
\end{ex}

\subsection{Line bundles on the flag variety}
Having located simple $G$-modules within the global sections of line bundles on $G/B^+$, it remains for us to construct and study those line bundles. Let $B$ be the Borel subgroup of $G$ corresponding to $-R_+$ (the opposite Borel subgroup to $B^+$). In the sequel, it will be more convenient to work with $G/B$ (which is isomorphic to $G/B^+$, since $B$ and $B^+$ are conjugate). Importantly, there is an open embedding
\begin{equation} \label{big}
\mathbb{A}^{|R_+|} \cong U^+ \hookrightarrow G/B, \quad u \mapsto uB/B,
\end{equation}
whose image is dense and often called the (opposite) open \textit{Schubert cell}. 

The following definition is really also a proposition; see \cite{jan}, I.5.8.

\begin{defi}
Let $H$ be a flat group scheme acting freely on a $k$-scheme $X$ in such a way that $X/H$ is a scheme; let $\pi: X \to X/H$ be the canonical morphism. There is an \textit{associated sheaf} functor 
$$\text{$\mathscr{L} = \mathscr{L}_{X,H}:$  \{$H$-modules\} $\to$ \{vector bundles on $X/H$\},}$$ defined on objects as follows: if $U \subseteq X/H$ is open, then
$$\mathscr{L}(M)(U) = \{ f \in \text{Hom}_\text{Sch}(\pi^{-1}(U),M_a): f(xh) = h^{-1} f(x) \}.$$
In case that $\pi^{-1}(U)$ is affine, these sections coincide with $(M \otimes k[\pi^{-1} U])^H$.
\end{defi}

The associated sheaf functor has a number of useful properties, including exactness. Much of its theoretical importance derives from its relation to \textit{induction}: whenever $H_2$ is a subgroup scheme of $H_1$ such that $H_1/H_2$ is a scheme, there is an isomorphism \begin{equation} \label{indsh}
R^n \text{ind}_{H_1}^{H_2} M \cong H^n(H_1/H_2, \mathscr{L}(M)), \quad n \ge 0,
\end{equation}
where $R^n$ refers to the $n$-th right derived functor. In fact, many results concerning induction are most readily proved geometrically via \eqref{indsh}. The interested reader is referred to \cite[\textsection I.5]{jan}.

\begin{notation}
For $\lambda \in \mathfrak{X}$, let $k_\lambda$ be the corresponding representation of $B$, arising from pullback along
$$B \to B/[B,B] \cong T.$$
Then define the sheaf $\mathcal{O}(\lambda) = \mathscr{L}_{G,B}(k_{-\lambda})$ on $G/B$. Because any character is of rank 1, $\mathcal{O}(\lambda)$ is a locally free sheaf of rank 1. (For more detail on this point, see \cite[I.5.16(2)]{jan} and \cite[II.1.10(2)]{jan}.)
\end{notation}

\begin{exercise}
Let $G = \text{SL}_2$ and let $\varpi$ denote the fundamental weight, i.e. the weight such that $\langle \varpi, \alpha^\vee \rangle = 1$, where $\alpha$ is a positive root. Verify that $\mathcal{O}(n \varpi)$ agrees with the invertible sheaf $\mathcal{O}(n) \in \mathbb{P}_k^1$. 
\end{exercise}

Restricting along the open embedding \ref{big}, we find
$$\Gamma(G/B,\mathcal{O}(\lambda)) \hookrightarrow \Gamma(U^+,\mathcal{O}_{U^+}) \cong k[U^+];$$
here we are using that line bundles on affine $k$-space, including $\mathcal{O}(\lambda)|_{U^+}$, are trivial.

\begin{prop} \leavevmode
\begin{enumerate}
    \item There is an action of $T$ on $\Gamma(U^+,\mathcal{O}_{U^+})$ such that $1$ has weight $\lambda$.
    \item The following are equivalent: 
    \begin{enumerate}
    \item $1$ extends to a section $v_\lambda \in \Gamma(G/B,\mathcal{O}(\lambda))$.
    \item $\lambda \in \mathfrak{X}_+$ is dominant.
    \item $\Gamma(G/B,\mathcal{O}(\lambda)) \ne 0$.
    \end{enumerate}
\end{enumerate}
\end{prop}

In light of the above, let us write $\nabla_\lambda = \Gamma(G/B,\mathcal{O}(\lambda))$ in case $\lambda \in \mathfrak{X}_+$. There is an inclusion
$$\nabla_\lambda^{U^+} \hookrightarrow k[U^+]^{U^+} = k \cdot 1,$$
which implies that $\nabla_\lambda$ is indecomposable and that it has a simple socle $$L_\lambda = \text{soc  $\nabla_\lambda$}.$$ Indeed, if there were a non-trivial decomposition
$$\nabla_\lambda = M \oplus N,$$
then we would obtain $\nabla_\lambda^{U^+} = M^{U^+} \oplus N^{U^+}$, which is at least two-dimensional by the non-triviality of each summand (see Prop. \ref{fixed}); similar considerations prove that the socle is simple. Now we are in a position to generalise our findings for $\text{SL}_2$.

\begin{thm}[Chevalley] \label{chev}
There is a bijection,
$$\mathfrak{X}_+ \rightarrow \{ \text{simple $G$-modules} \}/\cong, \quad \lambda \mapsto L_\lambda.$$
\end{thm}

\begin{exercise}
Prove Theorem \ref{chev} using the ideas in this section.
\end{exercise}

We end this section with some notation for future reference.

\begin{notation} \label{weyl}
For $\lambda \in \mathfrak{X}_+$, let $\Delta_\lambda = \nabla_{-w_0(\lambda)}^*$, where $w_0 \in W_{\text{f}}$ is the longest element. This is the \textit{Weyl module} associated to $\lambda$.
\end{notation}

\section{Kempf vanishing theorem}
\begin{defi}
Let $M$ be a finite-dimensional representation of $G$. We define the \textit{character} of $M$ to be
$$\text{ch $M$} = \sum_{\lambda \in \mathfrak{X}} (\text{dim $M_\lambda$}) e^\lambda \in \mathbb{Z}[\mathfrak{X}],$$
where $M_\lambda = \{ m \in M: \text{$tm = \lambda(t) m$ for all $t \in T$} \}$ is the $\lambda$-eigenspace of $M$ and $\mathbb{Z}[\mathfrak{X}]$ is the group algebra of $\mathfrak{X}$, written multiplicatively so that $e^\lambda e^\mu = e^{\lambda + \mu}$.
\end{defi}
The characters of the modules $\nabla_\lambda$ admit remarkably elegant expressions.

\begin{thm}[Weyl]
Let $W_{\text{f}}$ be the finite Weyl group of $G$ and $\rho = \frac{1}{2} \sum_{\alpha \in R_+} \alpha$. Then, for $\lambda \in \mathfrak{X}_+$,
$$\text{ch $\nabla_\lambda$}  = \frac{\sum_{w \in W} (-1)^{\ell(w)} e^{w(\lambda+\rho)}}{\sum_{w \in W} (-1)^{\ell(w)} e^{w \rho}}.$$
\end{thm}
At first glance, the right-hand side is only an element of the quotient field of $\mathbb{Z}[\mathfrak{X} + \mathbb{Z} \rho]$, but it turns out to lie in $\mathbb{Z}[\mathfrak{X}]$ and agree with the stated character. Many proofs exist for this formula in the case $k = \mathbb{C}$; see \cite[10.4]{hal} for one account. For arbitrary $k$, the result is derived as a consequence of the next theorem, which is of fundamental interest for us.

\begin{thm}[Kempf]
Let $\lambda \in \mathfrak{X}_+$. Then
$$H^i(G/B,\mathcal{O}(\lambda)) = 0$$
for all $i > 0$.
\end{thm}

We now sketch a proof of this theorem, assuming two black boxes; the original paper is \cite{kempf}, and another account is available in \cite[II.4]{jan}. To begin, we introduce some of the main characters in our story, the \textit{Steinberg modules}
\begin{equation} \label{stmod}
\text{St}_m = \nabla_{(p^m-1)\rho}, \quad m \ge 1.
\end{equation}
These are simple modules whose dimensions are $p^{m |R_+|}$. The fact that these induced modules are simple is our first black box; a beautiful proof is given in \cite{Kempf1981}.

Recall the Frobenius morphism Fr from the previous lecture, particularly $$\text{Fr}: G/B \to G/B.$$ The following isomorphism (our second black box) is due to Anderson \cite{and1} and Haboush \cite{hab}:
\begin{equation} \label{andhab}
\text{Fr}_*^m(\mathcal{O}((p^m-1)\rho)
\cong \text{St}_m \otimes \mathcal{O}.
\end{equation}
Now, for any $\gamma \in \mathfrak{X}_+$,
$$(\text{Fr}^m)^* \mathcal{O}(\gamma) = \mathcal{O}(p^m \gamma).$$
Using the projection formula \cite[Exercise II.8.3]{har} in combination with \eqref{andhab}, we find
$$(\text{Fr}^m)_*\mathcal{O}((p^m-1) \rho + p^m \gamma) = \text{Fr}_*^m \mathcal{O}((p^m-1) \rho) \otimes \mathcal{O}(\gamma) \cong \text{St}_m \otimes \mathcal{O}(\gamma).$$
Taking cohomology yields
$$H^i(G/B,\mathcal{O}((p^m-1)\rho + p^m \gamma)) \cong \text{St}_m \otimes H^i(G/B,\mathcal{O}(\gamma)).$$
But $\mathcal{O}(2 \rho)$ is ample, so by Serre's vanishing theorem \cite[III.5.2]{har}, the left-hand side is zero for $i \ne 0$ and sufficiently large $m$; hence the right-tensor factor on the right-hand side is necessarily zero.

\begin{exercise}
We have (a variant of) the Bruhat decomposition,
    $$G/B = \bigsqcup_{w \in W_{\text{f}}} B^+ \cdot xB/B,$$
    where each $B^+ \cdot xB/B \cong\mathbb{A}^{\ell(w_o) - \ell(x)}$. 
    \begin{enumerate}
        \item Use this decomposition to determine $\text{Pic}(G/B).$\footnote{Hint: You might want to abstract the properties of $G/B$. Suppose that $X$ is an algebraic variety containing an open dense affine space, whose complement is a union of divisors. What can you say about its Picard group?}
        \item Determine the class of $\mathcal{O}(\lambda)$ in the Picard group in terms of the previous description.
        \item All equivariant line bundles on $G/B$ have the form $\mathcal{O}(\lambda)$. Use this to determine when a line bundle on $G/B$ admits an equivariant lift, in terms of the root datum of $G$.
    \end{enumerate}
\end{exercise}

\makeatletter
\let\savedchap\@makeschapterhead
\def\@makeschapterhead{\vspace*{-1cm}\savedchap}
\chapter*{Lecture III}
\let\@makeschapterhead\savedchap

\section{Steinberg tensor product theorem}
\subsection{Motivation from finite groups}
Suppose momentarily that $G$ is a finite group with a normal subgroup $N$:
\begin{equation} \label{cliffseq}
1 \to N \to G \to G/N.
\end{equation}
Let $\sigma_g: G \to G$ denote conjugation by $g \in G$. Pulling back along $\sigma_g$ defines a functor $V \mapsto V^g$ on $G$-modules, whose image we call the \textit{twist} of $V$ by $g$.

Part of Clifford's theorem for finite groups states that if $V$ is a simple $G$-module, then $V|_N$ is a semi-simple $N$-module and all of its irreducible summands are $G$-conjugate \cite[\textsection 5.3]{webb}. With this fact in mind, let us take one additional assumption.

\begin{assumption} \label{cliff}
All simple $N$-modules extend to $G$-modules.
\end{assumption}

A consequence of Assumption \ref{cliff} is that every simple $N$-module $W$ is fixed by $G$, in the sense that $W^g \cong W$ for all $g \in G$. In particular, all the irreducible summands of $V|_N$ are isomorphic when $V$ is a simple $G$-module.

So, suppose in this setting that $V' \subseteq V$ is an irreducible summand of $V$ as an $N$-module. It then decomposes into copies of $V'$ with some multiplicity:
$$V \cong V' \oplus \cdots \oplus V'.$$
Then $$\text{Hom}_N(V',V) \otimes V' \to V, \quad f \otimes v' \mapsto f(v')$$
is an isomorphism of $G$-modules. (Indeed, it is easily seen to be surjective, and then we can compare dimensions.) Hence, in this scenario, we can conclude that every simple $G$-module arises as the tensor product of a simple $G/N$-module and an irreducible $N$-module extending to $G$. We will now witness a similar phenomenon in the setting of reductive groups.

\subsection{Back to reductive groups}
Let us return now to our usual level of generality, where $G$ is a reductive algebraic $k$-group for $k = \overline{k}$ of characteristic $p > 0$. Unless otherwise stated, the following assumption will be in force from here on.

\begin{assumption} \label{simple}
$G$ is semi-simple and simply connected.
\end{assumption}

We call a weight $\lambda \in \mathfrak{X}_+$ \textit{p-restricted} in case $\langle \lambda, \alpha^\vee \rangle < p$ for all simple roots $\alpha$; their subset is denoted $\mathfrak{X}_{< p} \subseteq \mathfrak{X}$. By our discussion in \textsection3, there is a Frobenius exact sequence
$$1 \to G_1 \to G \to G^{(1)} \to 1.$$
We would like to view this sequence as an analogue of \eqref{cliffseq}; in this light, the analogue of Assumption \ref{cliff} for reductive groups is the following result:

\begin{thm}[Curtis \cite{cur}]
If $\lambda \in \mathfrak{X}_{< p}$ then $L_\lambda|_{G_1}$ is simple, and moreover all simple $G_1$-modules occur in this way. Hence all simple $G_1$-modules extend to $G$.
\end{thm}

\begin{Eg}
The $p$ simple $\text{SL}_2$-modules $L_0, \dots, L_{p-1}$ remain simple when considered over $\mathfrak{g} = \mathfrak{sl}_2$.
\end{Eg} 

\begin{thm} \label{tens}
All simple $G$-modules are of the form $L_\lambda \otimes L_\mu^{(1)}$, for $\lambda \in \mathfrak{X}_{<p}$ and $\mu \in \mathfrak{X}_+$.
\end{thm}
Notice this theorem fits nicely in analogy to the conclusion of \textsection7.1: as there, it expresses simple $G$-modules as tensor products of simple modules over a quotient (namely $L_\mu^{(1)}$ over  $G/G_1 \cong G^{(1)}$) and simple modules over a normal subgroup which admit an extension to $G$ (namely $L_\lambda$ over $G_1$). By induction on Theorem \ref{tens}, we obtain a well-known and beautiful result:

\begin{thm}[Steinberg \cite{ste}]
Let $\lambda \in \mathfrak{X}_+$ and write $\lambda = \lambda_0 + p \lambda_1 + \cdots + p \lambda_m$, for $\lambda_i \in \mathfrak{X}_{<p}$. Then 
$$L_\lambda \cong L_{\lambda_0} \otimes L_{\lambda_1}^{(1)} \otimes \cdots L_{\lambda_m}^{(m)}.$$
\end{thm}
Importantly, it is a consequence of Assumption \ref{simple} that any $\lambda \in \mathfrak{X}_+$ admits the decomposition into $p$-restricted digits as described. 
\begin{remark}
One of the great uses of Steinberg's $\otimes$-theorem is that it reduces many questions (for instance, concerning characters) to a finite set of modules: the $L_\gamma$ for $\gamma \in \mathfrak{X}_{<p}$.
\end{remark}

\begin{Eg}
The theorem provides us a complete answer to the question of characters for $G = \text{SL}_2$. Define $$\text{Fr}: \mathbb{Z}[\mathfrak{X}] \to \mathbb{Z}[\mathfrak{X}], \quad e^\lambda \mapsto e^{p \lambda}.$$
Each $\lambda = n \in \mathbb{N}$ can be written $n = \sum_{i \ge 0} \lambda_i p^i$ with $0 \le \lambda_i < p$. Then, decomposing $L_n$ into a tensor product by Steinberg's theorem and taking characters, we obtain
$$\text{ch} \, L_n = \prod_{i \ge 0} \text{ch} \, L_{\lambda_i}^{(\text{Fr})^i} = \prod_{i \ge 0} (e^{-\lambda_i} + e^{-\lambda_i+2} + \cdots + e^{\lambda_i})^{(\text{Fr})^i}.$$
For instance, we have $p^m - 1 = (p-1) + (p-1)p + \cdots + (p-1)p^{m-1}$, so
\begin{align*}
    \text{ch} \, \text{St}_m = \text{ch} \, L_{p^m - 1} &= \left ( \frac{e^p - e^{-p}}{e-e^{-1}} \right ) \left ( \frac{e^p - e^{-p}}{e-e^{-1}} \right )^{(\text{Fr})} \cdots \left ( \frac{e^p - e^{-p}}{e-e^{-1}} \right )^{{(\text{Fr})}^{m-1}} \\
    &= \frac{e^{p^m} - e^{-p^m}}{e-e^{-1}};
\end{align*}
here we refer to the Steinberg module defined in \eqref{stmod}.
\end{Eg}

\begin{exercise} \label{simpchar}
Let $G = \text{SL}_2$.
    \begin{enumerate}
        \item Explicitly write out the characters of $L_m$ for $0 \le m \le p^2 - 1$.
        \item Hence express the characters of these $L_m$ in terms of the modules $\nabla_n$.
        \item Repeat this for $p^2$ and record your observations. What changes? 
        \item For bonus credit, repeat the first part for $0 \le m \le p^3 - 1$.
    \end{enumerate}
\end{exercise}
\section{Kazhdan--Lusztig conjecture}
Our next main goal is to state the Lusztig conjecture on $\text{ch} \, L_\lambda$. This formula was motivated by the earlier Kazhdan--Lusztig conjecture, which we will describe in this section after recalling certain elements of the theory of complex semi-simple Lie algebras.

\subsection{Background on complex semi-simple Lie algebras}
A comprehensive reference for this subsection is \cite{hum2}. Fix $\mathfrak{g}$ a complex semi-simple Lie algebra, containing
$$\mathfrak{h} \subseteq \mathfrak{b}^+ \supseteq \mathfrak{n}^+$$
Cartan and Borel subalgebras, along with its nilpotent radical, respectively. Recall that $\mathfrak{b}^+ \cong \mathfrak{h} \oplus \mathfrak{n}^+$ as vector spaces, that $\mathfrak{n}^+ = [\mathfrak{b},\mathfrak{b}]$, and that
$$\text{Hom}(\mathfrak{b}^+,\mathbb{C}) = \text{Hom}(\mathfrak{b}^+/[\mathfrak{b}^+,\mathfrak{b}^+],\mathbb{C}) = \mathfrak{h}^*.$$
To any $\lambda \in \mathfrak{h}^*$ we associate the \textit{standard} or \textit{Verma} $\mathfrak{g}$-module $\Delta_\lambda = U(\mathfrak{g}) \otimes_{U(\mathfrak{b}^+)} \mathbb{C}_\lambda$. By the PBW theorem, we can write $U(\mathfrak{g}) \cong U(\mathfrak{n}^{-}) \otimes U(\mathfrak{b}^+)$, so that
$$\Delta_\lambda \cong (U(\mathfrak{n}^{-}) \otimes U(\mathfrak{b}^+)) \otimes_{U(\mathfrak{b}^+)} \mathbb{C}_\lambda \cong U(\mathfrak{n}^{-}) \otimes \mathbb{C}_\lambda$$
as $\mathfrak{n}^{-}$-modules.

\begin{defi}[Bernstein--Gelfand--Gelfand]
The \textit{BGG category} $\mathcal{O}$ is the full subcategory of $\mathfrak{g}$-mod consisting of objects $M$ satisfying the following conditions:
\begin{itemize}
    \item $M$ is $\mathfrak{h}$-diagonalisable:
    $$M = \bigoplus_{\lambda \in \mathfrak{h}^*} M_\lambda,$$
    where $M_\lambda = \{ m \in M: \text{$h \cdot m = \lambda(h)m$ for all $h \in \mathfrak{h}$} \}$.
    \item $M$ is locally finite for the action of $\mathfrak{b}^+$: every $m \in M$ is contained in a finite-dimensional $\mathbb{C}$-vector space stable for the action of $\mathfrak{b}^+$.
    \item $M$ is finitely generated over $\mathfrak{g}$.
\end{itemize}
\end{defi}
Conspicuous objects in the category $\mathcal{O}$ are the $\Delta_\lambda$ and the finite-dimensional simple modules. In fact, $\Delta_\lambda$ has a unique simple quotient, since it has a unique maximal submodule; we denote this simple module by $L_\lambda$. 

\begin{prop}
There is a bijection,
$$\mathfrak{h}^* \to \{ \text{simple objects in $\mathcal{O}$} \}/\cong, \quad \lambda \mapsto [L_\lambda].$$
\end{prop}

\begin{Eg} \label{sl2s}
Consider the simplest non-trivial example, $\mathfrak{g} = \mathfrak{sl}_2$. We can be very explicit about the structure of the standard module $\Delta_\lambda$ in this case. It admits an infinite basis $$v_0 = 1 \otimes 1, \quad v_1 = f \otimes 1, \quad \dots, \quad v_m = \frac{1}{m!} f^m \otimes 1, \quad \dots$$ 
such that the action of $\mathfrak{sl}_2$ can be illustrated as follows:
\[
\begin{tikzcd}
\cdots \arrow[r,bend left,"\lambda-4"] & v_4 \arrow[l,bend left,swap,"5"] \arrow[r,bend left,"\lambda-3"] & v_3 \arrow[l,bend left,swap,"4"]
\arrow[r,bend left,"\lambda-2"] & v_2 \arrow[l,bend left,swap,"3"]
\arrow[r,bend left,"\lambda-1"] & v_1 \arrow[l,bend left,swap,"2"] \arrow[r,bend left,"\lambda"] & v_0 \arrow[l,bend left,swap,"1"]
\end{tikzcd}
\]
In terms of a standard basis $(h,e,f)$ for $\mathfrak{sl}_2$, arrows to the right represent the action of $e$, arrows to the left represent the action of $f$, and labels represent weights. It is visible from this description that if $\lambda \notin \mathbb{Z}_{\ge 0}$, then $\Delta_\lambda$ is simple; otherwise, if $\lambda \in \mathbb{Z}_{\ge 0}$, there is a short exact sequence 
$$0 \to L_{-\lambda - 2} \to \Delta_\lambda \to L_\lambda \to 0.$$
\end{Eg}

\begin{defi}
Recall the element $\rho = \frac{1}{2} \sum_{\alpha \in R_+} \alpha$. The \textit{dot action} of the finite Weyl group $W_{\text{f}}$ on $\mathfrak{h}^*$ (or on $\mathfrak{X}$) is given by $$x \bullet \lambda = x(\lambda + \rho) - \rho.$$
In words, this shifts the standard action of $W_{\text{f}}$ to have centre $-\rho$. Soon, we will use that the dot-action of $W_{\text{f}}$ on $\mathfrak{h}^*$ lifts to an action on the set of polynomial functions on $\mathfrak{h}^*$, which can be identified with $S(\mathfrak{h})$.
\end{defi}

Consider now the action of the universal enveloping algebra's centre $\mathcal{Z} = Z(U(\mathfrak{g}))$ on $\Delta_\lambda = U(\mathfrak{g}) \otimes_{U(\mathfrak{b}^+)} \mathbb{C}_\lambda$. Using that $\mathcal{Z}$ commutes with $U(\mathfrak{h}) \subseteq U(\mathfrak{g})$ and that $v_\lambda = 1 \otimes 1$ spans the $\lambda$-weight space of $\Delta_\lambda$, one can show that $\mathcal{Z}$ acts on $v_\lambda$ (and thus on $\Delta_\lambda = U(\mathfrak{g})v_\lambda)$ through a homomorphism $\chi_\lambda: \mathcal{Z} \to \mathbb{C}$; this map is known as a \textit{central character}. In fact $$\chi_\lambda(z) = \lambda(\pi(z)),$$ where $\pi: \mathcal{Z} \subseteq U(\mathfrak{g}) \to U(\mathfrak{h}) = S(\mathfrak{h})$ is the projection associated to the decomposition $\mathfrak{g} = \mathfrak{n}^{-} \oplus \mathfrak{h} \oplus \mathfrak{n}^+$ via the PBW theorem.
Now, the translation $$\mathfrak{h} \to \mathfrak{h}, \quad \lambda \mapsto \lambda - \rho$$ 
induces a $\mathbb{C}$-algebra automorphism $\sigma: S(\mathfrak{h}) \to S(\mathfrak{h})$. Modifying $\pi$ by $\sigma$ results in the \textit{twisted Harish-Chandra homomorphism} $\sigma \circ \pi: \mathcal{Z} \to S(\mathfrak{h})$, whose immense theoretical importance is suggested by the following theorem.

\begin{thm}[Harish-Chandra]
Let $\mathcal{Z} = Z(U(\mathfrak{g}))$.
\begin{enumerate}
    \item Consider $S(\mathfrak{h})^{(W_{\text{f}},\bullet)}$, the space of invariants in the dot-action of $W_{\text{f}}$ on $S(\mathfrak{h})$, and let $X = \mathfrak{h}^*/(W_{\text{f}},\bullet)$. The twisted Harish-Chandra homomorphism is a $\mathbb{C}$-algebra isomorphism $$\mathcal{Z} \to S(\mathfrak{h})^{(W_{\text{f}},\bullet)} \cong \mathbb{C}[X].$$
\item Every $\mathbb{C}$-algebra morphism $\chi: \mathcal{Z} \to \mathbb{C}$ is a central character $\chi = \chi_\lambda$, and $\chi_\lambda = \chi_\mu$ if and only if $\lambda$ and $\mu$ lie in the same $(W_{\text{f}},\bullet)$-orbit.
\end{enumerate}
\end{thm}

Given $M \in \mathcal{O}$ and a central character $\chi: \mathcal{Z} \to \mathbb{C}$, let $M^\chi$ denote the space of generalised $\chi$-eigenvectors for the action of $\mathcal{Z}$ on $M$. Since $M$ is generated by finitely many weight vectors, it can be shown that $M$ decomposes as the direct sum of finitely many $M^\chi$. Hence, in view of Harish-Chandra's theorem, the category $\mathcal{O}$ decomposes as
$$\mathcal{O} = \bigoplus_{\chi_\lambda} \mathcal{O}_{\chi_\lambda} = \bigoplus_{\lambda \in \mathfrak{h}^*/(W_{\text{f}},\bullet)} \mathcal{O}_\lambda,$$
where the \textit{block} $\mathcal{O}_\lambda$ consists of modules $M$ with $M = M^{\chi_\lambda}$. Through this decomposition of $\mathcal{O}$ and other considerations, particularly Jantzen's \textit{translation principle} (to be discussed soon in the analogous setting of algebraic groups), the problem of calculating characters in $\mathcal{O}$ can be reduced to the \textit{principal block} $\mathcal{O}_0$.

\begin{remark} \label{misnomer}
As used here, the term \textit{block} is a misnomer: blocks of a category are usually understood to be indecomposable, which need not hold for the $\mathcal{O}_\lambda$. In fact, the relevant condition for $\mathcal{O}_\lambda$ to be a genuine block is that $\lambda$ is integral.
\end{remark}

\subsection{The conjecture and its proof}
We will be prepared to state the Kazhdan--Lusztig conjecture after giving a final piece of notation.

\begin{notation}
For $x \in W_{\text{f}}$, let 
$$L_x = L_{xw_0 \bullet 0}, \quad \Delta_x = \Delta_{xw_0 \bullet 0} \in \mathcal{O}_0,$$
where $w_0 \in W_{\text{f}}$ is the longest element. For example, $L_{\text{id}} = L_{-2 \rho},$ $L_{w_0} = L_0$.
\end{notation}

\begin{conj}[Kazhdan--Lusztig] \label{klconj}
In the Grothendieck group of $\mathcal{O}$, we have
\begin{equation} \label{klc}
[L_x] = \sum_{y \in W_{\text{f}}} (-1)^{\ell(x) - \ell(y)} P_{y,x}(1) [\Delta_y],
\end{equation}
where the $P_{y,x} \in \mathbb{Z}[v,v^{-1}]$ are the Kazhdan--Lusztig polynomials.
\end{conj}

We omit a detailed introduction to the Kazhdan--Lusztig polynomials, instead directing the reader to \cite{soe}. In the sequel, we shall have occasion to refer to the \textit{Hecke algebra} $H = H(W_0,S_0)$ over $\mathbb{Z}[v,v^{-1}]$ associated to a Coxeter system $(W_0,S_0)$, with \textit{standard basis} $\{ h_w \}_{w \in W_0}$ and \textit{Kazhdan--Lusztig basis} $\{ b_w \}_{w \in W_0}$. An introduction to all these objects can be found in \cite{emtw}.

\begin{Eg}
Consider the example of $\mathfrak{sl}_2$. Here we have 
$$b_1 = h_1, \quad b_s = h_s + v,$$
so that $P_{1,1} = 1 = P_{s,s}$,  $P_{1,s} = v$, and $P_{s,1} = 0$ are the relevant Kazhdan--Lusztig polynomials. Hence the conjecture predicts
$$[L_{\text{id}}] = [\Delta_{\text{id}}], \quad [L_s] = -[\Delta_{\text{id}}]+[\Delta_s],$$
which we know by simplicity of the Verma module $\Delta_{w_0 \bullet 0} = \Delta_{-2}$ and by considering the exact sequence 
$$0 \to L_{-2} = L_{\text{id}} \to \Delta_0 \to L_0 = L_s \to 0$$
for the Verma module $\Delta_0 = \Delta_s$ (see Example \ref{sl2s}).
\end{Eg}

Let us make some remarks on the proof of Conjecture \ref{klconj}. Doing so will require us to work with perverse sheaves; we will omit a detailed description of this topic, introducing only the necessary notation.

\begin{notation} \label{icnote}
Suppose $Y = \bigsqcup_{\lambda \in \Lambda} Y_\lambda$ is a $\mathbb{C}$-variety stratified by subvarieties $Y_\lambda$ isomorphic to affine spaces. Then we have the following perverse sheaves on $Y$:
$$\Delta_\lambda^{\text{geom}} = j_{\lambda !} (\underline{\CM}_{Y_\lambda}[d_\lambda]), \quad \text{IC}_\lambda = j_{\lambda !*} (\underline{\CM}_{Y_\lambda}[d_\lambda]), \quad \nabla_\lambda^{\text{geom}} = j_{\lambda *} (\underline{\CM}_{Y_\lambda}[d_\lambda]),$$
where $j_\lambda: Y_\lambda \hookrightarrow Y$ is the inclusion, $d_\lambda$ is the complex dimension of $Y_\lambda$, and underlines denote constant sheaves. Referring instead to its support, $\text{IC}_\lambda$ is sometimes written as $\text{IC}(\overline{Y_\lambda})$.
\end{notation}

We will also need to recall briefly the notion of a differential operator on a commutative $k$-algebra $A$; see \cite[\textsection 15]{mcrob} for a more detailed exposition of this topic. The following is an inductive definition of differential operators on $A$.

\begin{defi}
A $k$-linear endomorphism $P \in \text{End}(A)$ is a \textit{differential operator of order $\le n \in \mathbb{Z}$} if either
\begin{enumerate}
    \item $P$ is a differential operator of order zero, i.e. multiplication by some $a \in A$.
    \item $[P,a]$ is a differential operator of order $\le n - 1$ for all $a \in A$.
\end{enumerate}
We write $D^n(A)$ for the ring of differential operators of order $\le n$, and $$D(A) = \bigcup_n D^n(A).$$
If $X$ is an affine $k$-scheme, we set $D(X) = D(k[X])$. For more general $k$-schemes $X$, this construction sheafifies to give a \textit{sheaf of differential operators} $\mathscr{D}_X$ on $X$.
\end{defi}

We are ready to return to the Kazhdan--Lusztig conjecture. Consider $Y = G/B$ over $\mathbb{C}$, stratified according to the Bruhat decomposition with $\Lambda = W_{\text{f}}$ and $Y_w = BwB/B$. In the Grothendieck group of $G/B$, the theory of perverse sheaves gives us a formula which is similar in appearance to \eqref{klc}:
\begin{equation} \label{ic}
[\text{IC}_x] = \sum_{y \in W_{\text{f}}} (-1)^{\ell(x)-\ell(y)} P_{y,x}(1) [\Delta_y^\text{geom}];
\end{equation}
here the ground ring is $A = \mathbb{C}$ and $\text{IC}_x = \text{IC}_x^{\mathbb{C}}$, etc. Equation \eqref{ic} is a consequence of Kazhdan--Lusztig's calculation of the stalks of intersection cohomology complexes via Kazhdan--Lusztig polynomials \cite{kl80}. 

On the other hand, the Beilinson--Bernstein localisation theorem (introduced in \cite{bb}) posits an equivalence of categories,
\begin{equation} \label{equiv1}
    (\text{$U(\mathfrak{g})/(Z^+))$-mod} \cong \text{$\mathscr{D}_{G/B}$-mod},
\end{equation}
where $Z^+$ is the kernel of the map $Z \to \text{End}(\mathbb{C})$ given by action on the trivial module. In one direction of this equivalence, we \textit{localise} modules for $U(\mathfrak{g})/(Z^+)$ to construct sheaves of $\mathscr{D}_{G/B}$-modules; in the other, we take global sections of $\mathscr{D}_{G/B}$-modules. Requiring certain good behaviour cuts out a \textit{regular holonomic} subcategory $\mathcal{H} \subseteq \mathscr{D}_{G/B}\text{-mod}$, and there is then an equivalence
\begin{equation} \label{equiv2}
\mathcal{H} \to \text{Perv}(G/B,\mathbb{C});
\end{equation}
this is a version of the Riemann--Hilbert correspondence, which in its classical form states that certain differential equations are determined by their monodromy. Under the composite of \eqref{equiv1} and \eqref{equiv2}, $L_x$ and $\Delta_x$ correspond to $\text{IC}_x$ and $\Delta_x^\text{geom}$, respectively, and hence we deduce the Kazhdan--Lusztig conjecture by combining \eqref{ic}, \eqref{equiv1}, and \eqref{equiv2}.

\section{Lusztig conjecture}
\subsection{Affine Weyl group}
The key point of this lecture is to state Lusztig's conjecture for the group $G$ over the field $k$, which parallels Conjecture \ref{klconj}. Recall the root system $(R \subseteq \mathfrak{X}, R^\vee \subseteq \mathfrak{X}^\vee)$ and the finite Weyl group $W_{\text{f}}$ introduced above.

\begin{defi} \leavevmode
\begin{enumerate}
    \item The \textit{affine Weyl group} of the dual root system $(R^\vee \subseteq \mathfrak{X}^\vee, R \subseteq \mathfrak{X})$ is 
    $$W = W_\text{f} \ltimes \mathbb{Z} R.$$
    We can realise $W$ as the subgroup of the affine transformations of $\mathfrak{X}_{\RM}$ generated by the elements
    $$s_{\alpha,m}: \mathfrak{X}_{\RM} \to \mathfrak{X}_{\RM}, \quad s_{\alpha,m}(\lambda) = \lambda - (\langle \lambda, \alpha^\vee \rangle - m) \alpha.$$
    Denote by $t_\gamma = s_{\gamma, 1} s_{\gamma,0} \in W$ the translation by $\gamma \in \mathbb{Z} R$.
    \item The \textit{$p$-dilated dot action} of $W$ on $\mathfrak{X}_{\mathbb{R}}$ is prescribed as follows:
    $$x \bullet_p \lambda = x(\lambda + \rho) - \rho, \quad t_\gamma \bullet_p \lambda = \lambda + p \gamma;$$
    here $x \in W_{\text{f}}$ and $\gamma \in \mathbb{Z} R$. In words, this action shifts the centre to $-\rho$ and dilates translations by a factor of $p$.
    \item The \textit{fundamental alcove} is
$$A_{\text{fund}} = \{ \lambda \in \mathfrak{X}_{\mathbb{R}}: \text{$0 < \langle \lambda + \rho, \alpha^\vee \rangle < p$ for all $\alpha \in R_+$} \}.$$
Its closure in $\mathfrak{X}_{\mathbb{R}}$ is a fundamental domain for the $p$-dilated action of $W$.
\end{enumerate}
\end{defi}

\begin{notation}
We let  $W^{\text{f}}$ and $^{\text{f}} W$ denote fixed sets of minimal coset representatives for $W/W_{\text{f}}$ and $W_{\text{f}} \backslash W$, respectively.
\end{notation}

\begin{Eg}
The following picture for $G = \text{SL}_2$ and $p = 5$ indicates some of the reflection 0-hyperplanes (points) for the dot action of $W$ on $\mathfrak{X}_{\mathbb{R}}$, along with the closure of the fundamental alcove (in blue). Notice the shift by $-\rho = -1$.
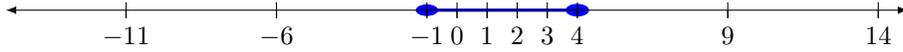
\vspace{0.3cm}
\begin{center}
\begin{tikzpicture}[xscale=2,yscale=1]
\draw[very thick, draw=blue] (-0.2,0) -- (0.8,0);
\path [draw=blue, fill=blue] (-0.2,0) circle (2pt);
\path [draw=blue, fill=blue, thick] (0.8,0.0) circle (2pt);
\draw[latex-latex] (-3,0) -- (3,0) ;
\foreach \x in  {-11,-6,-1,0,1,2,3,4,9,14}
\draw[shift={((\x/5,0)},color=black] (0pt,3pt) -- (0pt,-3pt);
\foreach \x in {-11,-6,-1,0,1,2,3,4,9,14}
\draw[shift={(\x/5,0)},color=black] (0pt,0pt) -- (0pt,-3pt) node[below] 
{$\x$};
\end{tikzpicture}
\end{center}
\end{Eg}

\begin{defi}
Suppose $\mathcal{A}$ is an abelian category. A \textit{Serre subcategory} of $\mathcal{A}$ is a non-empty full subcategory $\mathcal{C} \subseteq \mathcal{A}$ such that for any exact sequence
$$0 \to A' \to A \to A'' \to 0$$
in $\mathcal{A}$, $A \in \mathcal{C}$ if and only if $A', A'' \in \mathcal{C}$. Equivalently, $\mathcal{C}$ is closed under taking subobjects, quotients, and extensions in $\mathcal{A}$.
\end{defi}

\begin{prop}[Linkage Principle] \label{lp}
The category $\text{Rep}(G)$ is the direct sum of its \textit{blocks} $\text{Rep}_\lambda(G)$, for $\lambda \in \mathfrak{X}/(W,\bullet_p)$. Here $\text{Rep}_\lambda(G)$ is the Serre subcategory generated by simple modules $L_\mu$ for $\mu \in (W \bullet_p \lambda) \cap \mathfrak{X}_+$. 
\end{prop}

\begin{remark}
The appropriate analogue of Remark \ref{misnomer} applies here: our blocks $\text{Rep}_\lambda(G)$ need not be indecomposable. The abuse is not too severe for $p \ge h$, the Coxeter number, where $\text{Rep}_\lambda(G)$ is indecomposable unless $\langle \mu + \rho, \alpha^\vee \rangle$ is divisible by $p$ for every $\alpha \in R_+$. The `true' block decomposition is laid out in \cite{don}.
\end{remark}

If $p \ge h$, then $0$ is a \textit{$p$-regular} element in $A_{\text{fund}}$ (that is, it has trivial stabiliser under the $p$-dilated dot action of $W$). When $x$ is such that $x \bullet_p 0 \in \mathfrak{X}_+$, set
$$L_x = L_{x \bullet_p 0};$$
this is a simple module in the \textit{principal block} $\text{Rep}_0(G)$. We assign analogous meanings to $\nabla_x$ and $\Delta_x$. (Recall Notation \ref{weyl}.)

\begin{exercise} \label{resweights}
Suppose $p \ge h$. By explicit calculation, work out how many weights of the form $W \bullet_p 0$ are $p$-restricted for the root systems $A_1, A_2, B_2, G_2$. On the basis of these calculations, formulate a conjecture for the answer in general.\footnote{We will give the answer in a footnote on the final page of this lecture; the reader is warned to look only if they want to know!}
\end{exercise}

Now we are ready to state the Lusztig conjecture, at least in a simplified form.

\begin{conj}[Lusztig \cite{lus}]
Under certain assumptions on $p$ and $x$, the following equation holds in the Grothendieck group $[\text{Rep}_0(G)]$:
$$[L_x] = \sum_y (-1)^{\ell(x)+\ell(y)} P_{w_0 y, w_0 x}(1) [\Delta_y].$$
\end{conj}
The key point to note here is the independence of the formula from the prime $p$, subject to the assumptions mentioned; the parallel to the Kazhdan--Lusztig conjecture \ref{klconj} should also be apparent. In the next lecture we shall be explicit about these assumptions, and say more about the current status of this conjecture.

\subsection{Distribution algebras}
To conclude, we introduce distribution algebras and relate them to the representation theory of $G$. We will rely on this technology in later lectures, but it would be remiss to omit it entirely from our story. Moreover, consideration of distribution algebras allows one to see why the Linkage Principle holds (at least when $p$ is not too small).

\begin{defi}
Extending our previously discussed notion of left-invariant derivations on $G$, we define the $k$-algebra of $\textit{distributions}$ on $G$ to be
    $$\text{Dist $G$} = \text{left-invariant differential operators on $G$.}$$
\end{defi}
In terms of terminology already available to us, this is the most convenient definition of $\text{Dist}(G)$. For a definition and discussion of $\text{Dist}(G)$ as a subalgebra of $k[G]^*$, see \cite[\textsection I.7]{jan}. 

\begin{ex} \leavevmode
\begin{enumerate}
    \item Writing $k[\mathbb{G}_a] = k[z]$, $\text{Dist} \, \mathbb{G}_a$ has a countably infinite $k$-basis given by the divided powers $$\partial_z^{[n]} = \frac{\partial_z^n}{n!}, \quad n \ge 0,$$ where $\partial_z$ denotes differentiation by $z$.
    \item Writing $k[\mathbb{G}_m] = k[z,z^{-1}]$, $\text{Dist} \, \mathbb{G}_m$ has a $k$-basis in the elements $$\binom{\partial_z}{n} = \frac{\partial_z(\partial_z - 1) \cdots (\partial_z - n + 1)}{n!}, \quad n \ge 0.\footnote{In these examples, the numerators of the fractions are operators which send basis elements $z^m$ to multiples of $n!$, so ``division'' by $n!$ (which might be zero in $k$) is just a shorthand.}$$
\end{enumerate}
\end{ex}

The inclusion $\mathfrak{g} = \text{Lie}(G) \hookrightarrow \text{Dist}(G)$ induces an algebra homomorphism
$$\gamma: U(\mathfrak{g}) \to \text{Dist}(G).$$
For groups defined over a ground field of characteristic zero, $\gamma$ is an isomorphism; for $k$ of characteristic $p > 0$, all one can say is that $\gamma$ factors over an embedding
$$u(\mathfrak{g}) \hookrightarrow \text{Dist}(G).$$
In fact, $\text{Dist}(G)$ turns out to be the best replacement for $U(\mathfrak{g})$ in characteristic $p$. Any $G$-module $M$ gives rise to a locally finite $\text{Dist}(G)$-module, and the induced functor is fully faithful:
$$\text{Hom}_{G}(M,M') = \text{Hom}_{\text{Dist}(G)}(M,M').$$
Conversely, if $G$ is semi-simple and simply connected, then a theorem of Sullivan \cite{sull} establishes that any locally finite $\text{Dist}(G)$-module arises from a $G$-module.

We may assume that $G = G_k$ arises via base change from an algebraic group $G_\mathbb{Z}$ defined over $\mathbb{Z}$. Base extension to $\mathbb{C}$ yields $G_{\mathbb{C}}$, with Lie algebra $\mathfrak{g}_{\mathbb{C}}$. This Lie algebra is spanned by Chevalley elements $f_\alpha$, $e_\alpha$, and $h_\alpha = [e_\alpha,f_\alpha]$, $\alpha \in R_+$; see for instance \cite{hum1}. Consider the following $\mathbb{Z}$-subalgebra of $U(\mathfrak{g}_\mathbb{C})$:
$$U_{\mathbb{Z}} = \mathbb{Z} \left [ \frac{f_\alpha^\ell}{\ell!}, \binom{h_\alpha}{\ell}, \frac{e_\alpha^\ell}{\ell!} \right ].$$
We then have $\text{Dist $G_\mathbb{Z}$} = U_{\mathbb{Z}}$ and $\text{Dist $G_k$} = U_{\mathbb{Z}} \otimes_{\mathbb{Z}} k$. The algebra $U_{\mathbb{Z}}$ is known as a \textit{Kostant $\mathbb{Z}$-form} of $U(\mathfrak{g}_{\mathbb{C}})$.

\begin{remark}
If $Z(G)$ is reduced and $p$ is \textit{good} (in the sense of \cite[\textsection I.4.21]{jan}), the Linkage Principle \ref{lp} can be established by considering $Z(\text{Dist}(G))$. Since $\text{Dist}(G)$ replaces $U(\mathfrak{g})$ in characteristic $p$, this proof strategy is analogous to the usual approach to the Linkage Principle for category $\mathcal{O}$ of a complex semi-simple Lie algebra $\mathfrak{g}$. In that setting, consideration of central characters yields that $L_\lambda$ and $L_\mu$ are in the same block\footnote{Caution: do not forget about Remark \ref{misnomer}! A "block" for us is a subcategory $\mathcal{O}_\lambda$ of representations with central character $\chi_\lambda$, which is not necessarily indecomposable.} if and only if $\lambda = \mu$ in 
$$\mathfrak{h}^*/(W_{\text{f}},\bullet) = (\text{Spec $\mathcal{Z}$})(\CM),$$
where $\mathcal{Z} = Z(U(\mathfrak{g}))$. In characteristic $p$, the reduction modulo $p$ of $Z(U_{\mathbb{Z}})$ defines a subalgebra in $Z(\text{Dist $G$})$. Consideration of central characters for this subalgebra gives the analogous condition that $\lambda = \mu$ in 
$$\mathfrak{h}^*_{\mathbb{F}_p}/(W_{\text{f}},\bullet) = \mathfrak{h}^*/((W_{\text{f}},\bullet) \ltimes p \mathfrak{X}).$$ 
This is almost the Linkage Principle; to conclude, we need to pass from $(W_{\text{f}},\bullet) \ltimes p \mathfrak{X}$ to $(W_{\text{f}},\bullet) \ltimes p \mathbb{Z} R$. This is achieved by consideration of the centre $Z(G)$, which---if reduced---agrees as an algebraic group with the finite group $p\mathfrak{X}/p \mathbb{Z}R$. The proof of the Linkage Principle for small $p$, first provided by Andersen \cite{and}, has a rather different complexion.\footnote{As promised, the answer to Exercise \ref{resweights}: $|W_{\text{f}}|/\kappa$, where $\kappa = |\mathfrak{X}/\mathbb{Z}R|$ is the \textit{index of connection}.}
\end{remark}

\makeatletter
\let\savedchap\@makeschapterhead
\def\@makeschapterhead{\vspace*{-5cm}\savedchap}
\chapter*{Lecture IV}
\let\@makeschapterhead\savedchap

\section{Linkage and Blocks}
\subsection{Recollections} 
As previously, assume $G$ is a semi-simple and simply connected algebraic $k$-group. The affine Weyl group (of the dual root system) is $W = W_{\text{f}} \ltimes \mathbb{Z} R$. We also have the \textit{$p$-dilated affine Weyl group}, $W_p = W_{\text{f}} \ltimes p \mathbb{Z}R$, i.e.
\begin{align*}
    W_p = \langle \text{reflections in hyperplanes $\langle \lambda + \rho, \alpha^\vee \rangle = mp$, for $\alpha \in R$, $m \in \mathbb{Z}$} \rangle.
\end{align*}
Evidently the $p$-dilated dot action of $W$ corresponds to the regular dot action of $W_p$, so the choice to work with $\bullet_p$ or $W_p$ is mostly a matter of taste. We saw that a fundamental domain for the $(W_p,\bullet)$-action on $\mathfrak{X}_{\mathbb{R}}$ is the closure of
$$A_{\text{fund}} = \{ \text{$\lambda \in \mathfrak{X}_{\mathbb{R}}: 0 < \langle \lambda + \rho, \alpha^\vee \rangle < p$ for all $\alpha \in R_+$}  \}$$
in $\mathfrak{X}_{\mathbb{R}}$. Now $(W_p,S)$ is a Coxeter system, where $S = \{ \text{reflections in the walls of $A_{\text{fund}}$} \}$. We will assume $p \ge h$, the Coxeter number, so that $0 \in A_{\text{fund}}$ is a \textit{regular element} in the sense that $\text{Stab}_{(W_p,\bullet)}(0) = \{ 1 \}$. Recall that the \textit{facet} containing $\lambda \in \mathfrak{X}_{\mathbb{R}}$ is the subset of all $\mu \in \mathfrak{X}_{\mathbb{R}}$ sharing the same stabiliser as $\lambda$ under $(W,\bullet_p)$.

\begin{figure}[htp]
\centering
\setlength\weightLength{0.5cm}
\begin{tikzpicture}[auto,rotate=60]
\begin{rootSystem}{A}
\wt{-1}{-1}
\wt{0}{0}
\weightLattice{6}
\fill[gray!50,opacity=.5] \weight{-1}{-1} -- \weight{4}{-1} -- \weight{-1}{4} -- cycle;
\node [below] at (square cs:x=-1.5,y=-0.867) {\small\(-\rho\)};
\node [below=0.1cm] at (square cs:x=0,y=0) {\small\(0\)};
\draw[thick] \weight{0}{-6} -- \weight{0}{6};
\draw[thick] \weight{-6}{0} -- \weight{6}{0};
\draw[thick] \weight{6}{-6} -- \weight{-6}{6};
\end{rootSystem}
\end{tikzpicture}
\caption{Here is a picture for $G = \text{SL}_3$ and $p = 5$, with root hyperplanes in bold and $A_{\text{fund}}$ shaded. It is a good exercise to determine where the $p$-restricted weights $\mathfrak{X}_{<p}$ lie in this picture.}
    \label{fig:sl23}
\end{figure}
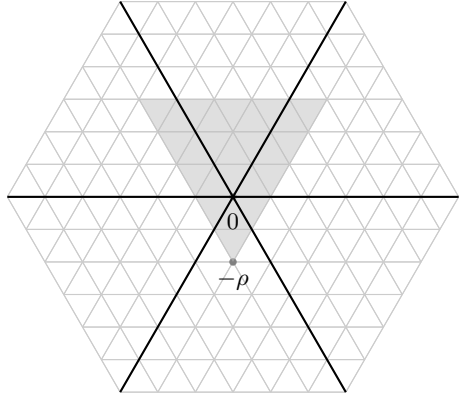
\subsection{Translation functors}
In Proposition \ref{lp}, we stated the Linkage Principle for $G$ by the decomposition
$$\text{Rep $G$} = \bigoplus_{\lambda \in \mathfrak{X}/(W_p, \bullet)} \text{Rep}_{\lambda}(G);$$
as usual, we call $\text{Rep}_0(G)$ the principal block. This decomposition implies other versions of the statement, such as to say that
$$\Ext^1_G(L_\lambda,L_\mu) = 0$$
if $\lambda, \mu \in \mathfrak{X}_+$ lie in different $(W_p,\bullet)$-orbits; see \cite[Corollary II.6.17]{jan}.

Most questions about the representation theory of $G$ can be reduced to questions about $\text{Rep}_0(G)$ using \textit{translation functors}; let us describe these briefly (for more detail, see \cite[\textsection II.7]{jan}). Given $\lambda \in \mathfrak{X}$, let $\text{pr}_\lambda$ denote the projection functor from $\text{Rep $G$}$ onto $\text{Rep}_{\lambda}(G)$; and given $\lambda, \mu \in A_{\text{fund}}$, let $\nu$ be the unique dominant weight in the $W$-orbit of $\mu - \lambda$. We then define the translation functor $T_{\lambda}^\mu: \text{Rep}(G) \to \text{Rep}(G)$ by the formula
$$T_{\lambda}^\mu(V) = \text{pr}_\mu(L_\nu \otimes \text{pr}_\lambda V).$$
This functor is exact and $(T_{\lambda}^\mu,T_{\mu}^\lambda)$ is an adjoint pair. By restriction, $T_{\lambda}^\mu$ induces a functor $\text{Rep}_\lambda(G) \to \text{Rep}_\mu(G)$. The \textit{translation principle} states this is an equivalence whenever $\lambda, \mu$ belong to the same facet. Roughly speaking, blocks associated to weights in the closure of a facet are ``simpler"; this is the essence of why considering the principal block is sufficient for many purposes.

\section{Elaborations on the Lusztig Conjecture}
\subsection{Explicit statement}
On our second pass, we will be precise in stating the Lusztig Conjecture. 

\begin{notation}
Write $L_x = L_{x \bullet 0}$, where $x \in {^\text{f}}W_p$ is a minimal coset representative, and similarly for $\nabla_x$ and $\Delta_x$.
\end{notation}

\begin{conj}[Lusztig]
Suppose $p \ge h$ and $x \bullet_p 0 \in \mathfrak{X}_+$, where $x \in W$ satisfies \textit{Jantzen's condition}: $\langle x \bullet_p 0 + \rho, \alpha^\vee \rangle \le p(p-h+2)$ for all $\alpha \in R_+$. Then
\begin{equation} \label{lcf} \tag{LCF}
[L_x] = \sum (-1)^{\ell(x) + \ell(y)} P_{w_0 y, w_0 x}(1)[\nabla_y],
\end{equation}
the sum running over $y \le x$ with $y \bullet_p 0 \in \mathfrak{X}_+$.
\end{conj}
The key feature to observe is the independence from $p$, or in other words that the formula is uniform over all $p \ge h$.

\subsection{History}
The conjecture was made in 1980 and proved in the mid 1990's for $p \ge N$, where $N$ is a non-explicit bound depending only the root system; this was work of Lusztig (\cite{lus94},  \cite{lus95}), Kashiwara--Tanisaki (\cite{kt95}, \cite{kt96}), Kazhdan--Lusztig (\cite{kl93}, \cite{kl94a}, \cite{kl94b}), and Andersen--Jantzen--Soergel (\cite{ajs}). In the mid-2000's, a new proof was provided by Arkhipov--Bezrukavnikov--Ginzburg \cite{abg04}. Early in the next decade, Fiebig gave another new proof \cite{fie11} and an explicit but enormous lower bound $N$ \cite{fie12}; for instance, $N = 10^{100}$ for $\text{GL}_{10}$. Most recently, the second author \cite{wil16b}, \cite{wil16c} (with help from Elias, He, Kontorovich, and McNamara variously in \cite{ew13}, \cite{hw15}, and the appendix to \cite{wil16c}) proved that the conjecture is not true for $\text{GL}_n$ for many $p$ on the order of exponential functions of $n$.

\begin{figure}[htp] 
    \centering
    $$
\begin{tabular}{ | c | c | c | c | c | c | c | c | } 
\hline
 & $\nabla_0$ & $\nabla_8$ & $\nabla_{10}$ & $\nabla_{18}$ & $\nabla_{20}$ & $\nabla_{28}$ & $\nabla_{30}$ \\ 
\hline
$L_0$ & \cellcolor{gray!25}1 & \cellcolor{gray!25} & \cellcolor{gray!25} & \cellcolor{gray!25} & \cellcolor{gray!25} & &  \\ 
\hline
$L_8$ & \cellcolor{gray!25}$-1$ & \cellcolor{gray!25}$1$ & \cellcolor{gray!25} & \cellcolor{gray!25} & \cellcolor{gray!25} & & \\ 
\hline
$L_{10}$ & \cellcolor{gray!25}1 & \cellcolor{gray!25}$-1$ & \cellcolor{gray!25}$1$ & \cellcolor{gray!25} & \cellcolor{gray!25} & & \\ 
\hline
$L_{18}$ & \cellcolor{gray!25}$-1$ & \cellcolor{gray!25}$1$ & \cellcolor{gray!25}$-1$ & \cellcolor{gray!25}$1$ & \cellcolor{gray!25} & &\\ 
\hline
$L_{20}$ & \cellcolor{gray!25}1 & \cellcolor{gray!25}$-1$ & \cellcolor{gray!25}$1$ & \cellcolor{gray!25}$-1$ & \cellcolor{gray!25}$1$ & & \\ 
\hline
$L_{28}$ & & & & & $-1$ & $1$ &\\ 
\hline
$L_{30}$ & & & & $-1$ & $1$ & $-1$ & $1$ \\ 
\hline
\end{tabular}
$$
    \caption{Plot of multiplicities in the principal block of $\text{SL}_2$ for $p = 5$. Shaded is the region in which Lusztig's conjecture is valid.}
    \label{fig:sl2l}
\end{figure}
\newpage

\begin{figure}[htp] 
    \centering
    \begin{minipage}{0.45\textwidth}
        \centering
        \begin{tikzpicture}[auto,rotate=60]
        \setlength\weightLength{1cm}
\begin{rootSystem}{A}
\weightLattice{2}
\wt{0}{0}
\fill[gray!50,opacity=.5] \weight{-1}{0} -- \weight{-1}{-1} -- \weight{0}{-1} -- cycle;
\node at (square cs:x=-1,y=-0.5) {\small\(-1\)}; 
\node at (square cs:x=-0.5,y=-0.25) {\small\(1\)};
\end{rootSystem}
\node[above,font=\large\bfseries] at (current bounding box.north) {$L_{(p-2) \rho}$};
\end{tikzpicture}
    \end{minipage}\hfill
    \begin{minipage}{0.45\textwidth}
        \centering
        \begin{tikzpicture}[auto,rotate=60]
        \setlength\weightLength{1cm}
\begin{rootSystem}{A}
\weightLattice{2}
\wt{0}{0}
\fill[gray!50,opacity=.5]
\weight{-1}{0} -- \weight{-1}{-1} -- \weight{0}{-1} -- cycle;
\node at (square cs:x=-1,y=-0.6) {\small\(2\)}; 
\node at (square cs:x=-0.5,y=-0.25) {\small\(-1\)};
\node at (square cs:x=0.5,y=0.25) {\small\(1\)};
\node at (square cs:x=0,y=0.6) {\small\(-1\)};
\node at (square cs:x=0,y=-0.5) {\small\(1\)};
\node at (square cs:x=0.5,y=-0.25) {\small\(-1\)};
\node at (square cs:x=-0.5,y=0.25) {\small\(1\)};
\end{rootSystem}
\node[above,font=\large\bfseries] at (current bounding box.north) {$L_{p \rho}$};
\end{tikzpicture}
    \end{minipage}
    \caption{Similar plots for $\text{SL}_3$ and the highest weights $(p-2) \rho$ and $p \rho$, respectively ($p \ge 3$). The gray regions are $A_{\text{fund}}$. The second example is the first one featuring a $2$; its multiplicities can be checked using the Steinberg tensor product theorem.}
    \label{fig:mults}
\end{figure}
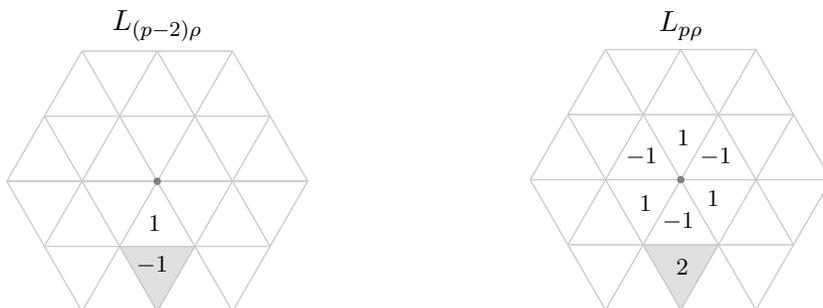

\begin{exercise}
Verify that the multiplicities given in Figure \ref{fig:mults} are correct.
\end{exercise}

\begin{exercise}
Let $W = \widetilde{A}_1$ denote the infinite dihedral group with Coxeter generators $s_0, s_1$, and let
    $$w_m = s_0 s_1 \cdots, \quad w_m' = s_1 s_0 \cdots$$
    be the elements given by the unique reduced expressions starting with $s_0$ and $s_1$, respectively, of length $m$. Note that any non-identity $x \in W$ is equal to a unique $w_m$ or $w_m'$.
    \begin{enumerate}
        \item Compute the Bruhat order on $W$.
        \item Prove inductively that one has
        $$b_{w_m} = h_{w_m} + \sum_{0 < n < m} v^{m-n} h_{w_n} + \sum_{0 < n' < m} v^{m-n'} h_{w'_{n'}} + v^m h_{\text{id}}.$$
        \item Deduce that Lusztig's character formula holds for $x \bullet_p 0 \in \mathfrak{X}_+$ if and only if $x \bullet_p 0$ has two $p$-adic digits. (Hint: This part will require use of Exercise \ref{simpchar}.)
    \end{enumerate}
\end{exercise}

\section{The Finkelberg--Mirkovi\'c  conjecture}
\subsection{Objects in affine geometry}
Let $G = G_k$ be a connected semi-simple group, obtained by base change from a group $G_{\mathbb{Z}}$ over $\mathbb{Z}$. Recall that the \textit{adjoint representation} of $G$ on $\mathfrak{g} = \text{Lie}(G)$ is given by differentiating inner automorphisms at the identity $e \in G$: $$\text{Ad}: G \to \mathfrak{gl}(\mathfrak{g}), \quad g \mapsto d(\text{Int $g$})_e,$$ 
where $\text{Int $g$}(h) = ghg^{-1}$ for $h \in G$. The \textit{adjoint group} of $G$ is then the image $$G_{\text{ad}} = \text{Ad $G$} \subseteq \text{Aut}(\mathfrak{g}).$$
\begin{exercise} \label{adj}
Show that the character and root lattices of $G_{\text{ad}}$ coincide. Hence deduce that for general semi-simple $G$, one always has an equivalence:
$$\text{Rep}_0(G) \cong \text{Rep}_0(G_{\text{ad}}).$$
\end{exercise}
For the sake of simplicity, and in light of Exercise \ref{adj}, we will be content to operate with the following assumption from now on.

\begin{assumption} \label{adjtyp}
$G$ is of \textit{adjoint type}, meaning that $\text{Ad}$ is faithful: $G \cong G_{\text{ad}}$.
\end{assumption}
The point of making this assumption is that Frobenius twist then yields a functor
$$(-)^{\text{Fr}}: \text{Rep $G$} \to \text{Rep}_0(G).$$
Indeed, the Frobenius twist of the simple module $L_\lambda$ is $L_{p\lambda}$, and $p\lambda \in p \mathfrak{X} = p \mathbb{Z} R$ by Assumption \ref{adjtyp}.

Let us denote by $G^\vee$ the dual group to $G$ over the \emph{complex} numbers. Let $F = \mathbb{C}((t))$ with ring of integers $\mathcal{O} = \mathbb{C}[[t]]$. Then
$$G^\vee(F) \supseteq K = G^\vee(\mathcal{O});$$
this $K$ is analogous to a maximal compact subgroup of $G^\vee(F)$. The assignment $t = 0$ defines an evaluation homomorphism
$$\text{ev}: K \to G^\vee(\mathbb{C}) = G^\vee;$$
consider then the preimage $\text{Iw} = \text{ev}^{-1}(B^\vee) \subseteq K$ of a Borel subgroup $B^\vee \subseteq G^\vee$. Now we can introduce some geometric objects: the \textit{affine flag variety} is
$$\text{Fl} = G^\vee(F)/\text{Iw} = \bigsqcup_{x \in W} \text{Fl}_x, \quad \text{where $\text{Fl}_x = \text{Iw} \cdot x \text{Iw}/\text{Iw}$},$$
which is a $K/\text{Iw} = G^\vee / B^\vee$-bundle over the \textit{affine Grassmannian}
$$\text{Gr} = G^\vee(F)/K = \bigsqcup_{x \in W^{\text{f}}} \text{Gr}_x, \quad \text{where $\text{Gr}_x = \text{Iw} \cdot x K/K$};$$
here we view $W^{\text{f}} \subseteq W$ as a set of minimal coset representatives for $W/W_{\text{f}}$. In these two decompositions, each Iw-orbit is isomorphic to an affine space of dimension $\ell(x)$; we refer to these orbits as \textit{Schubert cells}. 

Any $\lambda \in \mathfrak{X}$ corresponds to a cocharacter $\mathbb{G}_m \to T^\vee$, where $T^\vee \subseteq B^\vee$ is a maximal torus. We can then obtain a morphism $$F^\times = \mathbb{C}((t))^\times \to T^\vee(F),$$ sending $t$ to an element $t^\lambda \in T^\vee(F) \subseteq G^\vee(F)$. The $K$-orbits of the cosets $t^\lambda K \in \text{Gr}$ under the left action of $K$ are unions of Iw-orbits, and thus afford another (strictly coarser) stratification of Gr by \textit{spherical Schubert cells}:
$$\text{Gr} = \bigsqcup_{\lambda \in \mathfrak{X}} \text{Gr}_\lambda, \quad \text{where $\text{Gr}_\lambda =  K \cdot t^\lambda K$.}$$
The affine Grassmannian $\text{Gr}$ and affine flag variety $\text{Fl}$ are \textit{ind-varieties} (that is, colimits of varieties under closed embeddings). An in-depth treatment of their geometric properties would require at least another lecture; we recommend \cite{kumar}, \cite{br}, and \cite{zhu} for further information on this fascinating topic. 

\begin{Eg}
For $G = \text{SL}_2$, we have $G^\vee = \text{SL}_2$. The (complex points of the) affine Grassmannian can be written as the disjoint union
$$\mathbb{C}^0 \sqcup \mathbb{C}^1 \sqcup \mathbb{C}^2 \sqcup \mathbb{C}^3 \sqcup \mathbb{C}^4 \sqcup \cdots = \mathbb{C}^0 \sqcup (\mathbb{C}^1 \sqcup \mathbb{C}^2) \sqcup (\mathbb{C}^3 \sqcup \mathbb{C}^4) \sqcup \cdots$$
along complicated gluing maps. On the left-hand side, the indicated strata are the Schubert cells, which are in bijection with $W/W_{\text{f}} \cong \mathbb{Z}_{\ge 0}$; on the right-hand side, we have bracketed the spherical Schubert cells.
\end{Eg}

The following exercise is very beautiful and due to Lusztig \cite[\textsection 2]{lus81}.

\begin{exercise} \label{lexercise}
Let $V$ denote an $n$-dimensional $\mathbb{C}$-vector space and consider
    $$E = V^{\oplus n}$$
    equipped with the nilpotent operator
    $$t: E \to E, \quad (v_1, \cdots, v_n) \mapsto (0, v_1, \cdots, v_{n-1}).$$
    Denote by $Y$ the variety of $t$-stable $n$-dimensional subspaces of $E$.
    \begin{enumerate}
        \item Prove that $Y$ is a projective variety.
        \item Let $U \subseteq Y$ be the open subvariety of $t$-stable subspaces transverse to $V^{\oplus (n-1)} \oplus 0.$ Show that a point $X \in U$ is uniquely determined by maps $f_i: V \to V$, $1 \le i \le n -1$, such that
        $$X = \{ (f_{n-1}(v), f_{n-2}(v), \cdots, f_1(v), v): v \in V \}.$$
        \item Now use that $X$ is $t$-stable to deduce that $f_i = f_1^i$ and that $f_1^n = 0$. Conclude that 
        $$U \cong \mathfrak{N}(\text{End}(V)),$$
        the subvariety of nilpotent endomorphisms of $V$. 
        \item Prove $Y \cong \overline{\text{Gr}_{n \varpi_1}}$, a \textit{spherical Schubert variety} in the affine Grassmannian of $\text{GL}_n$. 
    \end{enumerate}
\end{exercise}

\subsection{Statement of the conjecture}
The following theorem is one of the most important geometric tools in the theory. Consider the constructible derived category $D^b_{(K)}(\text{Gr},k)$ (resp. $D^b_{(\text{Iw})}(\text{Gr},k)$), taking the stratification of Gr by $K$-orbits (resp. $\text{Iw}$-orbits), and its full subcategory of perverse sheaves $\text{Perv}_{(K)}(\text{Gr},k)$ (resp. $\text{Perv}_{(\text{Iw})}(\text{Gr},k)$).

\begin{thm}[Geometric Satake equivalence] \label{sat}
There is an equivalence of monoidal categories,
$$\text{Sat}: (\text{Rep}(G), \otimes) \to (\text{Perv}_{(K)}(\text{Gr},k),*),
$$
where $*$ is the \textit{convolution product} on perverse sheaves.
\end{thm}

This theorem was established by Mirkovi\'c--Vilonen \cite{mv07}. Their proof is non-constructive and relies on the Tannakian formalism: one shows that the category of perverse sheaves is Tannakian, and hence is equivalent to the representations of some group scheme. One then works hard to show this group scheme is $G$. In this way, an equivalence of categories is established without explicitly providing functors in either direction!

\begin{remarks} \leavevmode
\begin{enumerate}
    \item The geometric Satake equivalence can actually be used to \textit{construct} the dual group, without knowing its existence a priori.
    \item One often sees the theorem stated in terms of the $K$-equivariant category $\text{Perv}_{K}(\text{Gr},k)$, which is in fact equivalent to $\text{Perv}_{(K)}(\text{Gr},k)$. 
\end{enumerate}
\end{remarks}

\begin{notation}
From this point onward, we will sometimes refer to (co)standard and IC sheaves with coefficients in a general commutative ring $A$ (generalising Notation \ref{icnote}). If $A$ is not clear from context, we will use notation such as $\text{IC}_\lambda^A$ or $\text{IC}(\overline{Y_\lambda},A)$; commonly, $A$ will be $\ZM$, $\CM$, or $k$. 
\end{notation}

\begin{conj}[Finkelberg--Mirkovi\'c] \label{fm}
There is an equivalence of abelian categories $\text{FM}$ fitting into a commutative diagram:
\[
\begin{tikzcd}
\text{Rep}_0(G)  \arrow{r}{\cong}[swap]{\text{FM}} & \text{Perv}_{(\text{Iw})}(\text{Gr},k) \\
\text{Rep}(G) \arrow{u}{(-)^{\text{Fr}}} \arrow{r}{\cong}[swap]{\text{Sat}} & \text{Perv}_{(K)}(\text{Gr},k). \arrow{u}{\text{Forget}}
\end{tikzcd}
\]
Moreover, under the equivalence FM, 
$$L_x \mapsto \text{IC}_{x^{-1}}^k \quad \text{and} \quad \nabla_x \mapsto \nabla_{x^{-1}}^{\text{geom}}.$$
\end{conj}

\begin{assumption}
Through the remainder of these notes, we will assume Conjecture \ref{fm} holds; in fact, a proof was recently announced by Bezrukavnikov--Riche. The conjecture provides a useful guiding principle in geometric representation theory. All the consequences that we will draw from it below can be established by other means, but with proofs that are much more roundabout.
\end{assumption}

\begin{application}
As a first application, let us explain why the Finkelberg--Mirkovi\'c  conjecture helps us understand Lusztig's character formula \eqref{lcf}. Recall  that we want to find expressions of the form:
$$[L_x] = \sum a_{y,x} [\nabla_y].$$
If we apply the Finkelberg--Mirkovi\'c  equivalence, this becomes
$$[\text{IC}_{x^{-1}}^k] = \sum a_{y,x} [\nabla_{y^{-1}}^{\text{geom}}].$$
Taking Euler characteristics of costalks at $y^{-1} \text{Iw}/\text{Iw}$ yields
$$\chi((\text{IC}_{x^{-1}}^k)^{!}_{y^{-1}}) = (-1)^{\ell(y)} a_{y,x}.$$
Now, there exist ``integral forms'' $\text{IC}_{x^{-1}}^{\mathbb{Z}}$ such that the perverse shaves $\text{IC}_{x^{-1}}^{\mathbb{Z}} \otimes_{\mathbb{Z}}^L k$ are isomorphic to $\text{IC}_{x^{-1}}^k$
if (certain) stalks and costalks of the $\text{IC}_{x^{-1}}^{\mathbb{Z}}$ are free of $p$-torsion; suppose this holds and let $\text{IC}_{x^{-1}}^{\mathbb{Q}} = \text{IC}_{x^{-1}}^{\mathbb{Z}} \otimes_{\mathbb{Z}} \mathbb{Q}$. Then
\begin{align}
  (-1)^{\ell(y)} a_{y,x} = \chi((\text{IC}_{x^{-1}}^k)^{!}_{y^{-1}}) &= \chi((\text{IC}_{x^{-1}}^{\mathbb{Q}})^{!}_{y^{-1}}) \label{chi} \\
  &= (-1)^{\ell(x)} P_{y^{-1} w_0, x^{-1} w_0}(1), \label{chi0}
\end{align}
from which it follows that $a_{y,x} = (-1)^{\ell(x)+\ell(y)} P_{y^{-1} w_0, x^{-1} w_0}(1)$. Note it is the final equality on line \eqref{chi} which depends on the $p$-torsion assumption, while the equation on line \eqref{chi0} follows from a classical formula of Kazhdan--Lusztig \cite{kl80} for $P_{y,x}$ in terms of IC sheaf cohomology.

In conclusion, then, we can see that if $\text{IC}_{x^{-1}}^{\mathbb{Z}} \otimes_{\mathbb{Z}}^L k$ stays simple for all $x \in {^\text{f}}W_p$ satisfying Jantzen's condition, then the Lusztig conjecture holds. An induction shows that this implication is in fact an ``if and only if".
\end{application}

\makeatletter
\let\savedchap\@makeschapterhead
\def\@makeschapterhead{\vspace*{-1cm}\savedchap}
\chapter*{Lecture V}
\let\@makeschapterhead\savedchap

\section{Torsion explosion}
Assume $G$ is a Chevalley group scheme over $\mathbb{Z}$, with $k = \overline{k}$ of characteristic $p$ fixed as before. It is a 2017 result of Achar--Riche \cite{ar}, expanding on earlier work of Fiebig \cite{fie11}, that the Lusztig conjecture for $G_k$ is equivalent to the absence of $p$-torsion in the stalks and costalks of $\text{IC}(\overline{\text{Gr}_x},\mathbb{Z})$ for $x \in {^\text{f}}W_p$ satisfying Jantzen's condition. This provided a clear topological approach to deciding the validity of Lusztig's character formula; we will discuss this in some detail momentarily.

Let us first repaint the historical picture. In the mid-1990s, the character formula was proved for large $p > N$; this was work of many authors, continued into the late 2000s by Fiebig's discovery of an effective (enormous) bound for $N$ in terms of just the root system of $G$ \cite{fie11}. It remained to determine the soundness of stronger estimates for the best possible $N$ (e.g. linear or polynomial in $h$).

An important consequence of the topological formulation is the absence of $p$-torsion ($p > h$) in IC sheaves over spherical Schubert varieties lying inside $G^\vee/B^\vee$ (the \textit{finite flag variety}). This was first observed by Soergel in an influential paper \cite{soe00}; it follows by considering the ``Steinberg embedding'' associated to any dominant regular $\lambda \in \mathfrak{X}_+$,
$$G^\vee/B^\vee \hookrightarrow \text{Gr}, \quad g \mapsto g \cdot t^\lambda,$$
which is stratum-preserving and induces an equivalence of categories,
$$\text{Perv}_{(B^\vee)}(G^\vee/B^\vee) \cong \text{Perv}_{(\text{Iw})} \left ( U \right ),$$
where $U = \bigsqcup_{x \in W_{\text{f}}} \text{Iw} \cdot t^{x \lambda}$ (a locally closed subset of $\text{Gr}$). Using this property of the IC sheaves, the second author (in 2013, with help from several colleagues) was able to construct counter-examples to the expected bounds in Lusztig's conjecture and the James conjecture for irreducible mod $p$ representations of symmetric groups. In particular, torsion was shown to grow at least exponentially, as opposed to linearly (as implied by the Lusztig conjecture) or quadratically (as implied by the James conjecture); in other words, a phenomenon of ``torsion explosion''.

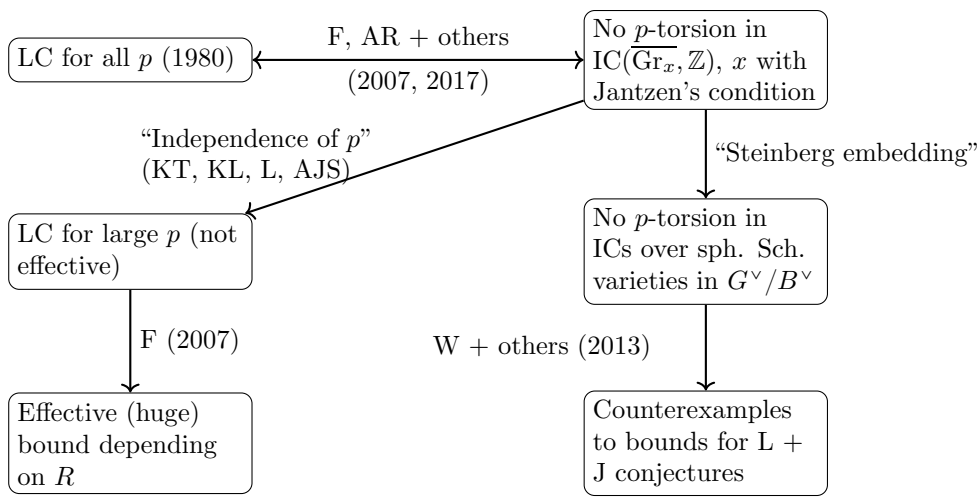
\hspace{0.5cm}
\begin{center}
\begin{tikzpicture}[scale=1.7, fatnode/.style={shape=rectangle,draw, text width=3cm,rounded corners}]
\node[fatnode] (A) at (-2.25, 1) {LC for all $p$ (1980)};
\node[fatnode] (B) at (2.25,1) {No $p$-torsion in $\text{IC}(\overline{\text{Gr}_x},\mathbb{Z})$, $x$ with Jantzen's condition};
\node[fatnode] (C) at (-2.25, -0.5) {LC for large $p$ (not effective)};
\node[fatnode] (D) at (2.25, -0.5) {No $p$-torsion in ICs over sph. Sch. varieties in $G^\vee/B^\vee$};
\node[fatnode] (E) at (2.25, -2) {Counterexamples to bounds for L + J conjectures};
\node[fatnode] (F) at (-2.25, -2) {Effective (huge) bound depending on $R$};
\draw [thick, <->] (A) -- (B) node[midway,above] {F, AR + others} node[midway,below] {(2007, 2017)};
\draw [thick, ->] (B) -- (C) node[midway,left, text width = 3.5cm] {``Independence of $p$'' (KT, KL, L, AJS)};
\draw [thick, ->] (B) -- (D)
node[midway,right]
{``Steinberg embedding''};
\draw [thick, ->]
(D) -- (E)
node[midway,left, text width = 3.5cm]
{W + others (2013)};
\draw [thick, ->]
(C) -- (F)
node[midway,right]
{F (2007)};
\end{tikzpicture}
\end{center}

\section{Geometric example}
In this section we discuss a simple geometric example, where the phenomenon of torsion in IC sheaves is clearly visible. For more details on this example, the reader is referred to \cite{jmw}.

Denote by $X$ the quadric cone 
\begin{align*}
    \mathbb{C}^2/(\pm 1) \cong \text{Spec} \, \mathbb{C}[X,Y]^{(\pm 1)} &= \text{Spec} \, \mathbb{C}[X^2,XY,Y^2] \\
    &= \text{Spec} \, \mathbb{C}[a,b,c]/(ab-c^2) \\
    &= \left \{ x = \begin{pmatrix}
    c & -a \\
    b & -c
  \end{pmatrix} \in \mathfrak{sl}_2(\mathbb{C}): \text{$x$ is nilpotent} \right \}.
\end{align*}
This variety has a stratification into two pieces, $X = X^{\text{reg}} \sqcup \{ 0 \}$. Its real points can be pictured as follows:

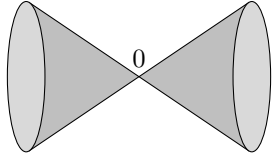
\begin{center}
\begin{tikzpicture}[node distance = 2cm, auto,rotate=90]
  \def\rx{1}    
  \def\ry{0.25}  
  \def\z{1.5}     

  \pgfmathparse{asin(\ry/\z)}
  \let\angle\pgfmathresult

  \coordinate (h) at (0, \z);
  \coordinate (k) at (0, -\z);
  \coordinate (O) at (0, 0);     
  \coordinate (A) at ({-\rx*cos(\angle)}, {\z-\ry*sin(\angle)});
  \coordinate (B) at ({\rx*cos(\angle)}, {\z-\ry*sin(\angle)});
  \coordinate (C) at ({-\rx*cos(\angle)}, {-\z+\ry*sin(\angle)});
  \coordinate (D) at ({\rx*cos(\angle)}, {-\z+\ry*sin(\angle)});

  \draw[fill=gray!50] (A) -- (O) node[align=center, above]{0} -- (B) -- cycle;
  \draw[fill=gray!30] (h) ellipse ({\rx} and {\ry});
  \draw[fill=gray!50] (C) -- (O) -- (D) -- cycle;
  \draw[fill=gray!30] (k) ellipse ({\rx} and {\ry});
\end{tikzpicture} 
\end{center}

Suppose $\mathcal{F}$ is a perverse sheaf on $X$ with respect to the given stratification. The following table indicates the degrees $i$ in which $\mathcal{H}^i(\mathcal{F}|_{X'})$ can be non-zero for a stratum $X' = X_{\text{reg}}$ or $X' = \{ 0 \}$.
$$
\begin{tabular}{ | c | c | c | c | c | c | } 
\hline
 & -3 & -2 & -1 & 0 & 1 \\ 
\hline
$X_{\text{reg}}$ & 0 & $\star$ & 0 & 0 & 0 \\ 
\hline
$\{0 \}$ & 0 & $\star$ & $\star$ & $\star$ & 0 \\ 
\hline
\end{tabular}
$$
To compute the intersection cohomology sheaf $\text{IC}(X,k) = \text{IC}(\overline{X_{\text{reg}}},k)$, we can use the \textit{Deligne construction}: 
$$\text{IC}(X,k) = \tau_{< 0} j_* \underline{k}_{X^{\text{reg}}}[2],$$
where $j: X^{\text{reg}} \hookrightarrow X$ is the inclusion and $j_*$ denotes the right-derived functor $Rj_*$. First compute 
$$(j_* \underline{k}_{X^{\text{reg}}}[2])_0 = \lim_{\varepsilon \to 0} H^{i + 2}(B(0,\varepsilon) \cap X^{\text{reg}},k);$$
since $B(0,\varepsilon) \cap X^{\text{reg}}$ is homotopic to $S_{\varepsilon}^3 \cap X^{\text{reg}} = S^3/( \pm 1 ) = \mathbb{RP}^3$, we reduce to computing $H^*(\mathbb{RP}^3,k)$, or, by the universal coefficient theorem, $H^*(\mathbb{RP}^3,\mathbb{Z})$:
$$H^*(\mathbb{RP}^3,\mathbb{Z}) = 
\begin{tabular}{ | c | c | c | c | } 
\hline
0 & 1 & 2 & 3 \\ 
\hline
$\mathbb{Z}$ & $0$ & $\mathbb{Z}/2\mathbb{Z}$ & $\mathbb{Z}$ \\ 
\hline
\end{tabular} \quad
\begin{tikzpicture}
\draw[-stealth,decorate,decoration={snake,amplitude=3pt,pre length=2pt,post length=3pt}] (0,-0.5) -- node[align=center, above] {UCF} ++(1,0);
\end{tikzpicture} \quad H^*(\mathbb{RP}^3,k) = \begin{tabular}{ | c | c | c | c | } 
\hline
0 & 1 & 2 & 3 \\ 
\hline
$k$ & $(k)_2$ & $(k)_2$ & $k$ \\ 
\hline
\end{tabular}
.$$
Here $(k)_2 = k$ if $2 = 0$ in $k$, and $0$ otherwise. Thus we find
$$\text{stalks of $j_* \underline{k}_{X^{\text{reg}}}[2]$} = \begin{tabular}{ | c | c | c | c | } 
\hline
-2 & -1 & 0 & 1 \\ 
\hline
$k$ & $0$ & $0$ & $0$ \\ 
\hline
$k$ & $(k)_2$ & $(k)_2$ & $k$ \\ 
\hline
\end{tabular}$$
and hence, applying $\tau_{<0}$:
$$\text{stalks of $\text{IC}(X,k)$} = \begin{tabular}{ | c | c | c | c | } 
\hline
-2 & -1 & 0 & 1 \\ 
\hline
$k$ & $0$ & $0$ & $0$ \\ 
\hline
$k$ & $(k)_2$ & $0$ & $0$ \\ 
\hline
\end{tabular}.$$
Similarly, $\text{IC}(X,\mathbb{Z}) = \underline{\mathbb{Z}}[2]$ and $\mathbb{D}(\text{IC}(X,\mathbb{Z})) = \text{IC}^+(X,\mathbb{Z})$\footnote{We note that $\text{IC}(X,\mathbb{Q})$ has two models over $\mathbb{Z}$, written $\text{IC}(X,\mathbb{Z})$ and $\text{IC}^+(X,\mathbb{Z})$; these are exchanged by Verdier duality. See \cite{jmw} or \cite{jut} for more on integral perverse sheaves.} have stalks as follows:
$$\text{stalks of IC} = \begin{tabular}{ | c | c | c | c | } 
\hline
-2 & -1 & 0 \\ 
\hline
$\mathbb{Z}$ & $0$ & $0$ \\ 
\hline
$\mathbb{Z}$ & $0$ & $0$ \\ 
\hline
\end{tabular} \quad
\begin{tikzpicture}
\draw[-stealth,decorate,decoration={snake,amplitude=3pt,pre length=2pt,post length=3pt}] (0,-0.5) -- node[align=center, above] {Verdier duality} ++(2,0);
\end{tikzpicture} \quad \begin{tabular}{ | c | c | c | c | } 
\hline
-2 & -1 & 0 \\ 
\hline
$\mathbb{Z}$ & $0$ & $0$ \\ 
\hline
$\mathbb{Z}$ & $0$ & $\mathbb{Z}/2\mathbb{Z}$ \\ 
\hline
\end{tabular} = \text{stalks of $\text{IC}^+$}.$$
Note particularly the 2-torsion in the lower right of the preceding table; this is what causes the aforementioned complications with torsion in this example. It turns out $\text{IC}(X,\mathbb{Z}) \otimes_{\mathbb{Z}}^L k$ is simple if the characteristic of $k$ is not 2; otherwise, it has composition factors $\text{IC}(X,k)$ and $\text{IC}(0,k)$. \begin{remark}
By Exercise \ref{lexercise}, $X$ occurs as an open piece of the spherical Schubert variety $\overline{\text{Gr}_{2 \varpi_1}} \subseteq \text{Gr}_{\text{SL}_2}.$ Under the geometric Satake correspondence, our above analysis then  translates into the fact that
$$\text{$\nabla_{2 \varpi_1}$ is irreducible} \quad \Leftrightarrow \quad p \ne 2.$$
On the other hand, under the Finkelberg--Mirkovi\'c  conjecture, 
$$\text{LCF for $L_{2p}$} \quad \Leftrightarrow \quad \text{$2p$ has $\le$ two $p$-adic digits} \quad \Leftrightarrow \quad \text{$\text{IC}(X,\mathbb{Z}) \otimes k$ is simple}.$$
The reader might like to check this, keeping in mind Exercise \ref{simpchar}.
\end{remark}

\section{Intersection forms}
A key reference for this section is \cite{dm}. Practically speaking, a substantial problem is that the Deligne construction cannot be computed except in the very simplest cases, but looking at resolutions provides a way forward.

In the case considered above, the Springer resolution is
$$f: \widetilde{X} = T^* \mathbb{P}^1(\mathbb{C}) \to X,$$
or pictorially:
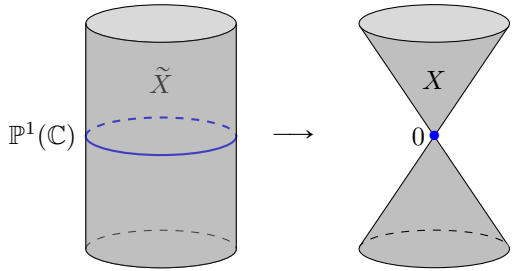
\begin{center}
\begin{tikzpicture}
\draw [fill=gray!30] (0,0) ellipse (1 and 0.25);
\draw (-1,0) -- (-1,-3);
\draw (0,-0.75) node[align=center]{$\widetilde{X}$};
\draw (-1,-1.5) node[align=center, left]{$\mathbb{P}^1(\mathbb{C})$};
\draw (1,-1.5) node[align=center, right]{\quad $\longrightarrow$};
\draw (-1,-3) arc (180:360:1 and 0.25);
\draw [dashed] (-1,-3) arc (180:360:1 and -0.25);
\draw [blue, thick] (-1,-1.5) arc (180:0:1 and -0.25);
\draw [dashed, blue, thick] (-1,-1.5) arc (180:360:1 and -0.25);
\draw (1,-3) -- (1,0);  
\draw [fill=gray,opacity=0.5] (-1,0) -- (-1,-3) arc (180:360:1 and 0.25) -- (1,0) arc (0:180:1 and -0.25);
\end{tikzpicture} \quad
\begin{tikzpicture}
  \def\rx{1}    
  \def\ry{0.25}  
  \def\z{1.5}     

  \pgfmathparse{asin(\ry/\z)}
  \let\angle\pgfmathresult

  \coordinate (h) at (0, \z);
  \coordinate (k) at (0, -\z);
  \coordinate (O) at (0, 0);     
  \coordinate (A) at ({-\rx*cos(\angle)}, {\z-\ry*sin(\angle)});
  \coordinate (B) at ({\rx*cos(\angle)}, {\z-\ry*sin(\angle)});
  \coordinate (C) at ({-\rx*cos(\angle)}, {-\z+\ry*sin(\angle)});
  \coordinate (D) at ({\rx*cos(\angle)}, {-\z+\ry*sin(\angle)});
  \draw[fill=gray!50] (A) -- (O) node[align=center, left]{0} -- (B) -- cycle;
  \draw[fill=gray!30] (h) ellipse ({\rx} and {\ry});
  \draw[fill=gray!50] (C) -- (O) -- (D) -- cycle;
  \fill [gray!50] (k) ellipse ({\rx} and {\ry});
  \draw [dashed] (-{\rx},-1.5) arc (180:360:{\rx} and -{\ry});
  \draw (-{\rx},-1.5) arc (180:0:{\rx} and -{\ry});
  \draw (0,0.75) node[align=center]{$X$};
  \draw (0,0) node[blue]{$\bullet$};
\end{tikzpicture} 
\end{center}
As we will see momentarily, there is an \textit{intersection form} $[-,-]$ on $$H^2(\PM^1) = \ZM[\PM^1\CM]$$ with values in $\mathbb{Z}$. Moreover,
$$\text{$[\mathbb{P}^1]^2$ invertible in $k$} \quad \Leftrightarrow \quad \text{$f_* \underline{k}_X [2]$ is semi-simple.}$$
In this case $[\mathbb{P}^1]^2 = -2$, so this falls in line with our earlier findings. (Recall: the self-intersection of any variety inside its cotangent bundle is the negative of its Euler characteristic!)

Let $X$ now be general, with stratification
$$X = \bigsqcup_{\lambda \in \Lambda} X_\lambda$$
and resolution $f: \widetilde{X} \to X$. We now have a schematic
\[
\begin{tikzcd}
\widetilde{X} \arrow{d}{f} & \widetilde{N_\lambda} \arrow{d}{} \arrow{l} & F_\lambda = f^{-1}(x_\lambda) \arrow{d} \arrow{l} \\
X = \bigsqcup_\lambda X_\lambda & N_\lambda \arrow{l} & \{ x_\lambda \} \arrow{l}
\end{tikzcd}
\]
where $x_\lambda \in X_\lambda$ is an arbitrary point and $\varnothing \ne N_\lambda \subseteq X$ is a \textit{normal slice} meeting the stratum $X_\lambda$ transversely at the point $x_\lambda$. Assume $f$ is \textit{semi-small}, so that
$$\text{dim $F_\lambda$} \le \frac{1}{2} \text{dim $\widetilde{N_\lambda}$} = n_\lambda.$$
This ensures the existence of a $\mathbb{Z}$-valued intersection form $IF_\lambda$ on top homology
$$H_{2n_\lambda}(F_\lambda) = \bigoplus_{C \in \mathcal{C}_\lambda} \mathbb{Z}[C],$$
where $\mathcal{C}_\lambda$ is the set of irreducible components of $F_\lambda$ of dimension $n_\lambda$.

\begin{prop}[\cite{jmw1}]
$f_* \underline{k} [\text{dim $\widetilde{X}$}]$ decomposes into a direct sum of IC sheaves if and only if every $IF_\lambda \otimes_{\mathbb{Z}} k$ is non-degenerate. 
\end{prop}

\begin{remark}
The $IF_\lambda$ are usually still difficult to calculate, since one must first find the fibres $F_\lambda$, compute components, and so on. A ``miracle situation'' arises when the $F_\lambda$ are smooth, since for $p: F_\lambda \to \text{pt},$ 
$$IF_\lambda = p_{!}(\text{Euler class of normal bundle of $F_\lambda$}).$$
\end{remark}

The following result underpins the idea of torsion explosion, by implying that torsion in the (co)stalks of spherical Schubert varieties in $\text{SL}_n/B$ grows at least exponentially with $n$.

\begin{thm}[Williamson \cite{wil16c}]
For any entry $\gamma$ of any word of length $\ell$ in the generators $\begin{pmatrix}
    1 & 1 \\
    0 & 1
  \end{pmatrix}$ and $\begin{pmatrix}
    1 & 0 \\
    1 & 1
  \end{pmatrix}$, one can associate a spherical Schubert variety $X_x \subseteq \text{SL}_{3 \ell +5}/B$, a Bott-Samelson resolution $\widetilde{X_x} \to X$, and a point $w_I \in X$, such that the miracle situation holds and the intersection form is $(\pm \gamma)$.
\end{thm}

For further discussion of the connections between torsion explosion and the bounds required for Lusztig's conjecture, see \cite[\textsection 2.7]{wil16a}.

\addtocontents{toc}{\protect\addvspace{1em}}

\bibliographystyle{alpha}
\bibliography{biblio.bib}
\end{document}